\documentclass[hidelinks,onefignum,onetabnum,nolinenumbers]{siamart}
\usepackage[T1]{fontenc}          
\usepackage{mathtools}             
\usepackage{bm}                    
\usepackage{mathrsfs}              
\usepackage{lipsum}
\usepackage{amsfonts}
\usepackage{graphicx,slashbox}
\usepackage{algorithm}
\usepackage{algpseudocode}
\newtheorem{example}{Example}
\newtheorem{assumption}[theorem]{Assumption}
\ifpdf
  \DeclareGraphicsExtensions{.eps,.pdf,.png,.jpg}
\else
  \DeclareGraphicsExtensions{.eps}
\fi

\newsiamremark{remark}{Remark}
\newsiamremark{hypothesis}{Hypothesis}
\crefname{hypothesis}{Hypothesis}{Hypotheses}
\newsiamthm{claim}{Claim}
\graphicspath{{./figs/}} 
\usepackage{algorithm}             
\usepackage{algpseudocode}         

\usepackage{graphicx}              
\usepackage{float}                 
\usepackage{subcaption}            
\usepackage{booktabs}              
\usepackage{multirow}              

\usepackage{enumerate}             
\usepackage{enumitem}              
\usepackage{setspace}              

\usepackage{array}                 
\newcommand{\tol}{\mathrm{tol}}    
\newcommand{\e}{\mathrm{e}}        
\newcommand{\x}{{\bm{x}}}          
\newcommand{\y}{{\bm{y}}}          
\newcommand{\Dc}{{\mathcal{D}}}    
\newcommand{\real}{\operatorname{Re}}

\headers{Efficient Diffusion-Generated Methods for BECs}{J. Guo, Y. Cai, and D .Wang}

\title{Efficient and stable diffusion generated methods for ground state computation in Bose--Einstein condensates\thanks{
Y. Cai is partially supported by the National Natural Science Foundation of China (Grant No. 12171041 and 11771036).
	D. Wang is partially supported by National Natural Science Foundation of China (Grant No. 12422116), Guangdong Basic and Applied Basic Research Foundation (Grant No. 2023A1515012199), Shenzhen Science and Technology Innovation Program (Grant No. JCYJ20220530143803007, RCYX20221008092843046), Guangdong Provincial Key Laboratory of Mathematical Foundations for Artificial Intelligence (2023B1212010001), and Hetao Shenzhen-Hong Kong Science and Technology Innovation Cooperation Zone Project (No.HZQSWS-KCCYB-2024016). }}
\author{Jing Guo\thanks{School of Mathematics and Statistics, Guangdong University of Technology, Guangdong, Guangzhou 510006, China (\email{jingguomath@gmail.com}).}\and
Yongyong Cai\thanks{Laboratory of Mathematics and Complex Systems (Ministry of Education), School of Mathematical Sciences, Beijing Normal University, Beijing,
	100875, PR China (\email{yongyong.cai@bnu.edu.cn}).}\and
Dong Wang  	\thanks{Corresponding author. School of Science and Engineering, The Chinese University of Hong Kong, Shenzhen, Guangdong 518172, China; Shenzhen International Center for Industrial and Applied Mathematics, Shenzhen Research Institute of Big Data, Guangdong 518172, China (\email{wangdong@cuhk.edu.cn}).}
}
\usepackage{amsopn}

\begin{document}

\maketitle

\begin{abstract}
This paper investigates numerical methods for approximating the ground state of Bose--Einstein condensates (BECs) by introducing two relaxed formulations of the Gross--Pitaevskii energy functional. These formulations achieve first- and second-order accuracy with respect to the relaxation parameter \( \tau \), and are shown to converge to the original energy functional as \( \tau \to 0 \). A key feature of the relaxed functionals is their concavity, which ensures that local minima lie on the boundary of the concave hull. This property prevents energy increases during constraint normalization and enables the development of energy-dissipative algorithms. Numerical methods based on sequential linear programming are proposed, accompanied by rigorous analysis of their stability with respect to the relaxed energy. To enhance computational efficiency, an adaptive strategy is introduced, dynamically refining solutions obtained with larger relaxation parameters to achieve higher accuracy with smaller ones. Numerical experiments demonstrate the stability, convergence, and energy dissipation of the proposed methods, while showcasing the adaptive strategy's effectiveness in improving computational performance.
\end{abstract}



\section{Introduction}

\sloppy{The phenomenon of Bose--Einstein condensation (BEC) marks a pivotal achievement in modern physics. It occurs when a system of bosons, cooled to ultra-low temperatures, collapses into the lowest quantum state, forming a macroscopic quantum system. Initially theorized by Bose and Einstein in the 1920s~\cite{Bose1924, Einstein1925} and experimentally realized in 1995 using dilute atomic gases~\cite{Anderson1995, Davis1995}, BECs exhibit remarkable properties such as superfluidity and quantum coherence. These unique characteristics have paved the way for groundbreaking applications across various scientific disciplines, including quantum simulations~\cite{Bloch2008} and precision measurements~\cite{Leggett2001}.}

The mathematical foundation of Bose--Einstein condensates is encapsulated by the Gross--Pitaevskii energy functional~\cite{BaoCai, Gross1961, Pitaevskii1961}, a fundamental tool for analyzing the properties of the condensate. This functional, defined for the condensate wave function \(\phi:=\phi(\x)\) (\(\x\in\Dc=\mathbb{R}^d, d=1,2,3\)), accounts for key physical contributions, including kinetic, potential, and interaction energies, and is represented  as
\begin{equation} \label{eng_orig}
	E(\phi) = \int_{\Dc } \left( \frac{1}{2} |\nabla \phi|^2 + V(\x) |\phi|^2 + \frac{\beta}{2} |\phi|^4  \right) \, d\x,
\end{equation}
where  \( V(\x) \) denotes the external trapping potential, and \( \beta \) represents the interaction strength between particles. The wave function $\phi$  satisfies the normalization condition
\begin{equation}\label{cons}
	\int_{\Dc} |\phi(\x)|^2 \, d\x = 1,
\end{equation}
ensuring that \(\phi\) is a valid probability density for the particles in the condensate. The corresponding minimization problem is given by
\begin{equation} \label{prob_orig}
	\min_{\|\phi\|_2 = 1} E(\phi),
\end{equation}
and its solution leads to the Euler--Lagrange equation, commonly referred to as the Gross--Pitaevskii equation (GPE):
\begin{equation}\label{eig_prob}
	\mu\phi(\x) = -\frac{1}{2}\Delta \phi + \left(V(\x) + \beta|\phi|^2 \right)\phi,
\end{equation}
where \(\mu\), the eigenvalue associated with the equation, is known as the chemical potential. The chemical potential can be computed as
\begin{equation}\label{mu}
	\mu = E(\phi) + \frac{\beta}{2}\int_{\mathbb{R}^d}|\phi|^4\,d\x.
\end{equation}
In practical computation,  the above whole space problem is usually truncated onto a bounded domain.
Under the assumption that \(V(\x)\) is a confining potential, i.e. \(\lim\limits_{|\x|\to\infty}V(\x)=+\infty\), as demonstrated in \cite[Theorem 2.4]{BaoCai}, the ground state of \eqref{prob_orig} exhibits exponential decay as \(|\x| \to \infty\), implying that its values become negligible at large distances from the origin. Consequently, it is possible to define a suitable bounded domain \(\Dc\) (still denoted as \(\Dc\), an interval in 1D, a rectangle in 2D and a box in 3D) for the Gross--Pitaevskii energy functional \eqref{eng_orig} and impose periodic boundary conditions without altering the fundamental properties of the problem. Hereafter, \(\Dc\subset\mathbb{R}^d\) is a bounded domain, and periodic boundary conditions are imposed for \(\phi(\x)\). Extensions to homogeneous Dirichlet boundary conditions or Neumann conditions are straightforward.

This mathematical framework has become indispensable for the study of BECs, providing a robust foundation for investigating their ground state properties and dynamics. It bridges theoretical predictions with experimental observations and has driven significant advancements in quantum science. Consequently, the computation of the ground state of BECs has been the focus of extensive research, leading to the development of numerous numerical methods. These methods can be broadly categorized into three main approaches:
\begin{itemize}
	\item Time-dependent methods based on gradient flow systems derived from the Schrödinger equation.
	\item Optimization algorithms for minimizing the energy functional~\eqref{eng_orig} under the \( L^2 \)-norm constraint.
	\item Eigenvalue solvers for the time-independent GPE~\eqref{eig_prob}.
\end{itemize}

Among these, gradient flow methods are the most widely employed. These include the continuous normalized gradient flow~\cite{BaoDu04,BaoWang07,Wang14}, the discretized normalized gradient flow~\cite{BaoCai,BaoDu04,CaiLiu,BaoJak}, and the Riemannian gradient flow~\cite{YinHuaZha,YinHuaCai}. In particular, the second type of gradient flow system, which incorporates a projection step to enforce normalization, is preferred for its simplicity and efficiency. Temporal discretization of this system has led to widely used numerical schemes such as the forward Euler method, backward Euler method, and semi-implicit methods, as well as Lagrange multiplier-based approaches for normalization~\cite{LiuCai}.

Alternatively, the Riemannian gradient flow system takes advantage of the Riemannian geometry of the \( L^2 \)-norm constraint. By utilizing the Riemannian gradient as the descent direction, these methods ensure energy dissipation along the tangent space of the Riemannian manifold. The forward Euler method applied to the Riemannian gradient flow system has been successfully used to compute both ground and excited states of BECs~\cite{YinHuaZha,YinHuaCai}.

Optimization-based methods also play a significant role in minimizing the energy functional while preserving the \( L^2 \)-norm constraint. These methods are particularly effective for finding local minima corresponding to the ground state of the condensate. Techniques such as the preconditioned Riemannian conjugate gradient method~\cite{ShuTang} and the Riemannian Newton method~\cite{JiaAndWen,TianCai} have been proposed to ensure normalization throughout the iterative process. Lastly, eigenvalue solvers, such as the \( J \)-method~\cite{RobPatDan}, finite element methods~\cite{ChenGongHeZhou,ChenHeZh}, and mixed finite element techniques~\cite{GalHauk}, have been extensively employed to solve the time-independent GPE~\eqref{eig_prob}, where the eigenvalue corresponds to the system's chemical potential.

Despite these advancements, challenges remain in achieving efficient, accurate, and energy-dissipative solutions for the ground state of BECs. In this work, we address these challenges from an optimization perspective, introducing a novel relaxation-based framework for solving the Gross--Pitaevskii energy minimization problem. Specifically, we propose two types of relaxed energy functionals, demonstrating that as the relaxation parameter \( \tau \to 0 \), these functionals converge to the original energy functional~\eqref{eng_orig}. A key feature of the relaxed functionals is their concavity, which prevents energy increases during constraint normalization and enables the development of robust, energy-dissipative algorithms.

The two relaxed energy functionals differ in their accuracy: the first achieves first-order accuracy in the relaxation parameter \( \tau \), while the second achieves second-order accuracy. Building on the relaxation framework introduced in~\cite{Wan22} for Dirichlet partition problems, we adapt and extend this approach to BECs, enhancing its accuracy and applicability. The concavity of the relaxed functionals simplifies the optimization process and ensures energy dissipation, leading to efficient algorithms based on sequential linear programming.

To further enhance computational efficiency, we introduce an adaptive strategy for dynamically selecting the relaxation parameter \( \tau \). This strategy balances computational performance with solution quality, significantly improving the efficiency of the proposed methods. Numerical experiments validate the stability, convergence, and energy dissipation of the algorithms, confirming the effectiveness of the relaxed framework and the adaptive \( \tau \) strategy.

The remainder of this paper is organized as follows. Section~\ref{sec:relaxed_functionals} introduces the formulation of the relaxed energy functionals for non-rotating BECs and discusses their concavity properties. It also presents the proposed algorithms and the adaptive \( \tau \) strategy, emphasizing their energy-dissipative behavior. Section~\ref{sec:numerical_results} provides numerical results to demonstrate the performance of the proposed methods. In Section~\ref{sec:rot_BECs}, the relaxation technique is extended to rotating BECs, with numerical experiments validating its effectiveness. Finally, conclusions are drawn in Section~\ref{sec:conclusion}.
\section{Relaxation of Gross--Pitaevskii energy functional}\label{sec:relaxed_functionals}

This section introduces two distinct formulations of the relaxed Gross--Pitaevskii energy functional, which approximate the original energy functional with first-order and second-order accuracy, respectively. The discussion focuses on analyzing the properties of the relaxed energy functional and establishing the existence of local minima for the relaxed problem. Additionally, numerical algorithms based on sequential linear programming are presented, along with a detailed analysis of their energy stability. To begin, we establish the following assumptions:

\begin{assumption}\label{assum:BEC}
	The parameters involved in the minimization problem satisfy the following conditions:
	\begin{enumerate}
		\item \( \beta \) is a real, positive constant.
		\item The potential \( V \) satisfies \( V(\x) \in L^\infty(\Dc) \) and \( V(\x) \geq 0 \) for all \( \x \in \Dc \).
	\end{enumerate}
\end{assumption}
As demonstrated in the introduction part, the relaxed problems discussed in the following sections will  be considered on the bounded domain \(\Dc\) with periodic boundary conditions. It is clear that \(\Delta\) generates a semi-group on \(L^2(\Dc)\), denoted as \(e^{t\Delta}\) (\(t\ge0\)). The \(L^2\) inner product reads \[
\langle \phi_1, \phi_2 \rangle = \operatorname{Re} \left( \int_{\Dc} \phi_1(\x) \overline{\phi_2}(\x)\, d\x \right),
\]
where \(\overline{\phi}_2\) is the complex conjugate of \(\phi_2\). We now proceed to derive the relaxed energy functionals.

\subsection{First-order accurate relaxation}  

To construct a first-order accurate relaxation scheme, we begin by rewriting the expression \( e^{\tau\Delta}\phi \) in an alternative form. For \( \phi \in L^2(\mathcal{D}) \), the function \( e^{\tau\Delta}\phi \), which represents the solution at time \( \tau \) to the heat equation with analytical initial condition \( \phi \), can be expressed as  
\begin{equation*}  
	e^{\tau\Delta}\phi = \sum_{k = 0}^{\infty}\frac{(\tau\Delta)^k}{k!}\phi.
\end{equation*}  
Using this expansion, the inner product \( \langle e^{\tau\Delta}\phi, \phi \rangle \) can be written as  
\begin{equation}  
	\langle e^{\tau\Delta}\phi, \phi \rangle = \langle \phi, \phi \rangle + \langle \tau\Delta\phi, \phi \rangle + \left\langle \sum_{k = 2}^{\infty} \frac{(\tau\Delta)^k}{k!}\phi, \phi \right\rangle.  
\end{equation}  
For \( \phi \in H^1(\mathcal{D}) \) with periodic boundary conditions, this yields 
\begin{equation} \label{app_GinzbOrd1}  
	-\int_{\mathcal{D}} \frac{1}{2} (\Delta \phi) \bar{\phi} \, d\x = \frac{\|\phi\|_2^2}{2\tau} - \frac{1}{2\tau} \int_{\mathcal{D}} \left| e^{\frac{\tau}{2} \Delta}\phi \right|^2 d\x + O(\tau).  
\end{equation}  
Substituting this result into the energy functional \eqref{eng_orig}, we obtain the {\bf first-order accurate relaxed energy functional}
\begin{equation} \label{eng_concOrd1}  
	E^{1,\tau}(\phi) = \frac{1}{2\tau} + \int_{\mathcal{D}} \left( -\frac{1}{2\tau} \left| e^{\frac{\tau}{2} \Delta}\phi \right|^2 + V(\x) |\phi|^2 + \frac{\beta}{2} |\phi|^4  \right) \, d\x ,  
\end{equation}
subject to the constraint \( \|\phi\|_2 = 1 \).  Denote 
\begin{equation} \label{Space_SOrd1}
	\mathbb{S} := \left\{ \phi \mid \|\phi\|_2 = 1, \, E^{1,\tau}(\phi) < \infty \right\},
\end{equation}  
The minimization problem for the relaxed energy functional \eqref{eng_concOrd1} can then be formulated as
\begin{equation} \label{prob_relxOrd1}
	\min_{\phi \in \mathbb{S}} E^{1,\tau}(\phi).
\end{equation}

Applying Lemma~\ref{lem:exisAbsPhiOrd1} to \eqref{eng_concOrd1convx} in Appendix~\ref{sec:append} establishes the existence and uniqueness of a solution to the minimization problem: $\min_{\phi \in \mathbb{S}} E^{1,\tau}(|\phi|)$.

Combining this result with Lemma~\ref{lem:EconvOrd1absphi}, it follows that the problem $\min_{\phi \in \mathbb{S}} E^{1,\tau}(\phi)$
also admits a solution, and the positive minimizer is unique.  
The following theorem gives a formal statement of this result:
\begin{theorem}\label{thm:EconcOrd1_exis}[Existence of the local minimum]
	Let $\tau > 0$. Under Assumption~\ref{assum:BEC}, the minimization problem \eqref{prob_relxOrd1} has a unique minimizer up to a constant phase factor. That is, if $u$ and $v$ are both minimizers of \eqref{prob_relxOrd1}, then there exists a constant $\theta \in \mathbb{R}$ such that $v = e^{i\theta} u$.
\end{theorem}
\begin{proof}
See the detailed proof in Appendix~\ref{sec:append}.
\end{proof}

\subsection{Algorithm for the first-order accurate relaxed problem}\label{sec:algOrd1}
To construct an algorithm for solving the first-order relaxed problem, we begin by establishing convergence properties. Applying formula~\eqref{app_GinzbOrd1} yields the following  result:
\begin{lemma}\label{lem:conv_EOrd1}
	For any \(\phi \in H^1(\Dc)\), it holds that
	\begin{equation}
		\lim_{\tau \to 0} E^{1,\tau}(\phi) = E(\phi).
	\end{equation}
\end{lemma}
Lemma~\ref{lem:conv_EOrd1} ensures that, as $\tau \to 0$, the minimizer of the relaxed problem \eqref{prob_relxOrd1} converges to the minimizer of the original problem \eqref{prob_orig}. Next, we show that the local minimizer of \eqref{prob_relxOrd1} is uniformly bounded.
\begin{lemma}\label{lem:phiLinf}
	Under Assumption~\ref{assum:BEC}, the minimizer $\phi$ of the problem \eqref{prob_relxOrd1} satisfies $\phi \in L^\infty(\mathcal{D})$, and $\|\phi\|_\infty\leq M$, where $M=\sqrt{\frac{\beta+2\int_{\Dc}V(\x)d\x}{\beta|\Dc|}}$ ($|\Dc|$ the Lebesgue measure of $\Dc$). 
\end{lemma}

\begin{proof}  
	The minimizer $\phi$ of the problem \eqref{prob_relxOrd1} satisfies the Euler-Lagrange equation: 
	\begin{equation}\label{EL_eq}  
		\lambda(\phi) \phi= -\frac{1}{2\tau} e^{\tau \Delta} \phi + \big(V(\x) + \beta |\phi|^2 \big) \phi,  
	\end{equation}  
	where the Lagrange multiplier $\lambda(\phi)$ is given by  
	\begin{equation}\label{EL_LagMP}  
		\lambda(\phi) = \int_{\mathcal{D}} \left(-\frac{1}{2\tau} \big| e^{\frac{\tau}{2} \Delta} \phi \big|^2 + V(\x) |\phi|^2 + \beta |\phi|^4 \right) d\x < \infty.  
	\end{equation}  
	Rearranging \eqref{EL_eq}, we obtain 
	\[ \lambda(\phi) \phi+   \frac{1}{2\tau}e^{\tau \Delta} \phi = \big(V(\x) + \beta |\phi|^2 \big) \phi.  \]
	Taking the $L^\infty$ norm on both sides yields  
	$$  
	\left\| \lambda(\phi)\phi + \frac{1}{2\tau} e^{\tau \Delta} \phi \right\|_\infty = \|V(\x) + \beta |\phi|^2 \|_\infty \|\phi\|_\infty.  
	$$  
	Using $
	\| e^{\tau \Delta} \phi \|_\infty \leq \|\phi\|_\infty,  
	$
	it follows that  
	$$  
	\left| \frac{1}{2\tau} + \lambda(\phi) \right| \|\phi\|_\infty \geq \|V(\x) + \beta |\phi|^2 \|_\infty \|\phi\|_\infty.  
	$$  
	Moreover,
\[
\frac{1}{2\tau} + \lambda(\phi) = E^{1,\tau}(\phi) + \int_\Dc \frac{\beta}{2} |\phi|^4 \, d\x>E^{1,\tau}(\phi)  > 0,
\]
where the positivity of \( E^{1,\tau}(\phi) \) follows from Lemma~\ref{lem:convex}.  
	From the definition of $\phi \in \mathbb{S}$ in \eqref{Space_SOrd1}, we know that $\|\phi\|_\infty > 0$, and hence we have
	\[ \frac{1}{2\tau} + \lambda(\phi)  - \|V(\x) + \beta |\phi|^2 \|_\infty \geq 0. \]
	Combining this with $V(\x) \geq 0$ and $\beta > 0$ (as stated in Assumption~\ref{assum:BEC}), we deduce  
	$$  
	\frac{1}{2\tau} + \lambda(\phi) - \beta \| |\phi|^2 \|_\infty \geq 0,  
	$$  
	which implies  
	$$  
	\|\phi \|_\infty \leq \sqrt{\frac{ \frac{1}{2\tau} + \lambda(\phi)}{\beta}}.  
	$$  
	Noticing that $\lambda(\phi)+\frac{1}{2\tau}=E^{1,\tau}(\phi)+\frac{\beta}{2}\int_{\Dc}|\phi|^4d\x\leq 2E^{1,\tau}(\phi)$,
	testing $\tilde{\phi}=1/\sqrt{|\Dc|}$, we find $E^{1,\tau}(\phi)\leq E^{1,\tau}(\tilde{\phi})=\frac{1}{|\Dc|}\left(\frac{\beta}{2}+\int_{\Dc}V(\x)d\x\right)$ and the conclusion follows.
\end{proof}

Lemma~\ref{lem:phiLinf} establishes that all minimizers $\phi$ of problem~\eqref{prob_relxOrd1} satisfy the uniform bound $\|\phi(\x)\|_\infty \leq M$ for some positive constant $M$.  
This result allows us to modify the quartic nonlinearity term $|\phi|^4$ in \eqref{eng_concOrd1} by introducing the truncated nonlinear term
\begin{equation}\label{trunc_nonlin}
	\widetilde{F}(\phi) = 
	\begin{cases}
		6M^2|\phi|^2-8M^3\phi+3M^4, & \phi > M, \\[.5em]
		|\phi|^4, & |\phi| \leq M, \\[.5em]
		6M^2|\phi|^2+8M^3\phi+3M^4, & \phi<- M.
	\end{cases}
\end{equation}
The corresponding truncated energy functional is defined as
\begin{equation}\label{eng_truncOrd1}
	\widetilde{E}^{1,\tau}(\phi) = \frac{1}{2\tau} + \int_{\mathcal{D}} \left(
	-\frac{1}{2\tau} \left| e^{\frac{\tau}{2} \Delta}\phi \right|^2 + V(\x) |\phi|^2 + \frac{\beta}{2} \widetilde{F}(\phi) - \kappa |\phi|^2 \right)\,
	d\x + \kappa,
\end{equation}
where $\kappa$ is a positive constant, which will be used in the construction of the algorithm. Note that the choice of the truncated nonlinear term \eqref{trunc_nonlin} ensures that the Fr\'echet differentiability of the energy functional remains unchanged. Moreover, the solutions to the corresponding Euler--Lagrange equation include the critical points of the problem \eqref{prob_relxOrd1}.

\begin{lemma}\label{lem:min_equivalence}
	For $\tau > 0$, under Assumption~\ref{assum:BEC}, the following equality holds:
	\begin{equation}\label{equi_TruncOptProb}
		\min_{\lVert \phi \rVert_2 = 1} \widetilde{E}^{1,\tau}(\phi) = \min_{\lVert \phi \rVert_2 = 1} E^{1,\tau}(\phi),
	\end{equation}
	where $E^{1,\tau}(\phi)$ and $\widetilde{E}^{1,\tau}(\phi)$ are defined in \eqref{eng_concOrd1} and \eqref{eng_truncOrd1}, respectively.
\end{lemma}
\begin{proof}
	First, from the definition of $\widetilde{E}^{1,\tau}(\phi)$ in \eqref{eng_truncOrd1}, we observe that any minimizer $\phi^*$ of $\min_{\lVert \phi \rVert_2 = 1} E^{1,\tau}(\phi)$ must also minimize $\min_{\lVert \phi \rVert_2 = 1} \widetilde{E}^{1,\tau}(\phi)$. That is,
	$$
	\phi^* = \arg \min_{\lVert \phi \rVert_2 = 1} E^{1,\tau}(\phi) \implies \phi^* = \arg \min_{\lVert \phi \rVert_2 = 1} \widetilde{E}^{1,\tau}(\phi).
	$$
	Additionally, the truncation in \eqref{trunc_nonlin} ensures that $\widetilde{E}^{1,\tau}(|\phi|)$ is convex with respect to $|\phi|$. This convexity guarantees the uniqueness (up to a phase factor) of the positive minimizer. Hence, the minimizer $|\phi^*|$ is unique and globally minimizes $\min_{\substack{ \lVert \phi \rVert_2 = 1}} \widetilde{E}^{1,\tau}(|\phi|)$. Consequently, we obtain
	$$
	\min_{\substack{ \lVert \phi \rVert_2 = 1}} \widetilde{E}^{1,\tau}(|\phi|) = \min_{\substack{ \lVert \phi \rVert_2 = 1}} E^{1,\tau}(|\phi|).
	$$
	Since both $E^{1,\tau}(\phi)$ and $\widetilde{E}^{1,\tau}(\phi)$ are invariant under phase transformations, the equality \eqref{equi_TruncOptProb} holds.
\end{proof}

\begin{lemma}\label{lem:EconcOrd1_conc}
	For $\tau > 0$ and $\kappa \geq \|V(\x)\|_\infty + 3\beta M^2$, under Assumption~\ref{assum:BEC}, the functional $\widetilde{E}^{1,\tau}(\phi)$ defined in \eqref{eng_truncOrd1} is concave on the set $\{\phi \mid \|\phi\|_2 \leq 1\}$.
\end{lemma}

\begin{proof}
	Computing the second variation of \(E^{1,\tau}(\phi)\), we obtain
	\begin{equation*}
		\begin{split}
			\nabla^2 \widetilde{E}^{1,\tau}(\phi)&[\eta, \eta] = \real \left[\int_\Dc \left(-\frac{1}{\tau} \left\vert \e^{\frac{\tau}{2}\Delta} \eta \right\vert^2 - 2(\kappa - V(\x)) \odot (\eta \bar{\eta}) \right)\, d\x \right. \\[.5em]
			& + \int_{|\phi(\x)| \le M} \left( 4\beta|\phi|^2 \odot (\eta \bar{\eta}) + 2\beta \phi^2 \odot (\bar{\eta} \bar{\eta}) \right) d\x \left. - \int_{|\phi(\x)| \ge M} 6\beta M^2|\eta|^2 \, d\x \right] \\[.5em]
			&\le \real \left( -2\int_{\Dc} \left( \kappa - V(\x) - 3\beta \min\left\{\Vert \phi \Vert_\infty^2, M^2\right\} \right) \odot (\eta \bar{\eta}) \, d\x \right),
		\end{split}
	\end{equation*}
	where \(\real(\cdot)\) denotes the real part, and \(\odot\) represents the Hadamard product.
	By selecting \(\kappa\) such that
	\[
	\kappa \ge \Vert V(\x) \Vert_\infty + 3\beta \min\left\{\Vert \phi \Vert_\infty^2, M^2\right\},
	\]
	we ensure
	$
	\nabla^2 \widetilde{E}^{1,\tau}(\phi)[\eta, \eta] \le 0.
	$
	This confirms the concavity of \(\widetilde{E}^{1,\tau}(\phi)\) on \(\left\{\phi \mid \Vert \phi \Vert_2 \le 1\right\}\), completing the proof.
\end{proof}

For a non-constant concave functional defined on a closed set, the local minimum, if it exists, must lie on the boundary of the set. This property enables us to demonstrate the equivalence of the minimization problems over the unit ball and the unit sphere, as stated in the following corollary.
\begin{corollary}\label{Cor:EconcOrd1_SphereBall}
	For $\tau > 0$, \(\kappa\ge \Vert V(\x) \Vert_\infty + 3\beta M^2\), under Assumption~\ref{assum:BEC}, the following equality holds
	\begin{equation}\label{equi_OpProb}
		\min\limits_{\Vert \phi\Vert_2 \leq 1} \widetilde{E}^{1,\tau}(\phi) = \min\limits_{\Vert \phi\Vert_2=1} \widetilde{E}^{1,\tau}(\phi),
	\end{equation}
	where $\widetilde{E}^{1,\tau}(\phi)$ is defined in  \eqref{eng_truncOrd1}.
	
\end{corollary}
Combining  Lemma~\ref{lem:min_equivalence} with Corollary~\ref{Cor:EconcOrd1_SphereBall}, we obtain
\[
\min_{\lVert \phi \rVert_2 = 1} E^{1,\tau}(\phi)=\min_{\Vert \phi\Vert_2 \leq 1} \widetilde{E}^{1,\tau}(\phi),
\]
where $E^{1,\tau}(\phi)$ and $\widetilde{E}^{1,\tau}(\phi)$ are defined in \eqref{eng_concOrd1} and \eqref{eng_truncOrd1}, respectively.
This implies that solving the relaxed optimization problem \eqref{prob_relxOrd1} reduces to computing $\min_{\Vert \phi\Vert_2 \leq 1} \widetilde{E}^{1,\tau}(\phi)$, which can be achieved via the sequential linear programming approach:
\begin{equation}\label{prob_SLPA}
	\phi^{n+1} = \arg \min_{\|\phi\|_2 \leq 1} \widetilde{E}^{1,\tau}(\phi^n) + \langle \nabla \widetilde{E}^{1,\tau}(\phi^n), \phi - \phi^n \rangle,
\end{equation}
where
$$
\nabla \widetilde{E}^{1,\tau}(\phi) = -\frac{1}{\tau} e^{\tau \Delta} \phi + 2V(\x)\phi + \frac{\beta}{2} \widetilde{f}(\phi) - 2\kappa \phi,
$$
and $\widetilde{f}(\phi)$ is defined as
\begin{equation}\label{trunc_nonlinDerv}
	\widetilde{f}(\phi) := \frac{d\widetilde{F}(\phi)}{2d\phi} = 
	\begin{cases}
		12M^2\phi - 8M^3, & \phi > M, \\[.5em]
		4|\phi|^2\phi, & |\phi| \leq M, \\[.5em]
		12M^2\phi + 8M^3, & \phi < -M.
	\end{cases}
\end{equation}
The solution to \eqref{prob_SLPA} can be explicitly expressed as follows:
\begin{lemma}\label{lem:EconcOrd1solSLPA}
	For \(\tau > 0\),   the solution of \eqref{prob_SLPA} is
	\begin{equation}\label{num_solOrd1}
		\phi^{n+1} = \frac{\frac{1}{\tau}\e^{\tau \Delta} \phi^n - 2V(\x)\phi^n-\frac{\beta}{2}  \widetilde{f}(\phi^n) + 2\kappa \phi^n}{\left\Vert \frac{1}{\tau}\e^{\tau \Delta} \phi^n - 2V(\x)\phi^n-\frac{\beta}{2}  \widetilde{f}(\phi^n) + 2\kappa \phi^n \right\Vert_2},
	\end{equation}
	where $n\ge 0$, $\widetilde{f}(\cdot)$ is given in \eqref{trunc_nonlinDerv}.
\end{lemma}

\begin{proof}
	Since the first term in \eqref{prob_SLPA} is constant, it reduces to
	\[
	\arg \min_{\Vert \phi \Vert_2 \leq 1} \left\langle \nabla \widetilde{E}^{1,\tau}(\phi^n), \phi \right\rangle.
	\]
	Using the expression for \( \nabla \widetilde{E}^{1,\tau}(\phi^n) \), the minimizer is given by \eqref{num_solOrd1}.
	This completes the proof.
\end{proof}

\begin{theorem}\label{thm:EconcOrd1_disp}
	For \(\tau > 0\), \(\kappa\ge \Vert V(\x) \Vert_\infty + 3\beta M^2\), the numerical energy functional updated using \eqref{num_solOrd1} satisfies
	\begin{equation*}
		\widetilde{E}^{1,\tau}(\phi^{n+1}) \le \widetilde{E}^{1,\tau}(\phi^n),\quad n\ge 0,
	\end{equation*}
	where $\widetilde{E}^{1,\tau}(\cdot)$  is defined in \eqref{eng_truncOrd1}.
\end{theorem}
\begin{proof}
	By applying Lemma~\ref{lem:EconcOrd1solSLPA}, we obtain 
	\begin{equation}
		\widetilde{E}^{1,\tau}(\phi^n) \ge \widetilde{E}^{1,\tau}(\phi^n) + \left\langle \nabla \widetilde{E}^{1,\tau}(\phi^n), \phi^{n+1} - \phi^n \right\rangle.
	\end{equation}
	Next, as established in Lemma~\ref{lem:EconcOrd1_conc}, the functional $\widetilde{E}^{1,\tau}(\phi)$ is concave when \(\kappa\ge \Vert V(\x) \Vert_\infty + 3\beta M^2\). This concavity implies
	\begin{equation}
		\widetilde{E}^{1,\tau}(\phi^n) + \left\langle \nabla \widetilde{E}^{1,\tau}(\phi^n), \phi^{n+1} - \phi^n \right\rangle \ge \widetilde{E}^{1,\tau}(\phi^{n+1}).
	\end{equation}
	Combining these inequalities, we deduce that
	$
	\widetilde{E}^{1,\tau}(\phi^{n+1}) \le \widetilde{E}^{1,\tau}(\phi^n).$
	This completes the proof.
\end{proof}

The result obtained in Theorem~\ref{thm:EconcOrd1_disp} establishes that the numerical energy functional decreases with each iteration of the algorithm. In the following, we describe the algorithm used to compute the local minimum of the relaxed problem. This algorithm iteratively updates the approximation of the solution by solving a sequential linear programming problem. 
\begin{algorithm}[H]
	\caption{Algorithm for computing the local minimum of \eqref{prob_relxOrd1}} \label{Alg:EconcOrd1}
	\begin{algorithmic}
		\State {\bf Input:} Let $\Dc$ be a given bounded domain, $\tau > 0$, $\tol > 0$, $V(\x)$,  $\kappa > 0$, $N_{\max}$ be the maximum number of outer iterations, and $\phi^0$ be the initial condition.
		\State {\bf Output:} $\phi^{n+1}$, the numerical approximation of the local minimum of \eqref{prob_relxOrd1}.
		
		\State Initialize $n = 0$.
		\State Update $\phi^{n+1}$ by \eqref{num_solOrd1}.
		
		\If{the stopping criterion is met}
		\State Terminate the iteration.
		\EndIf
		
		\State Set $n = n + 1$.
		\State \textbf{End While}
	\end{algorithmic}
\end{algorithm}
\begin{remark}
	The term \( \e^{\tau \Delta}\phi^n \) in \eqref{num_solOrd1} can be efficiently computed using the fast Fourier transform (FFT).  The computational complexity of each iteration is \( O(N\log(N)) \), where \( N \) denotes the number of spatial grid points.   The unconditional energy decaying property of $\widetilde{E}^{1,\tau}$ for sufficiently large $\kappa$ ensures convergence, which leads to the minimizer of $E^{1,\tau}$.
\end{remark}


\subsection{Second order accurate relaxation}  \label{sec:algOrd2}
In this subsection, we focus on deriving the second-order accurate relaxed energy functional and its associated optimization algorithm. To achieve this, we begin by expanding \( e^{c\tau \Delta}\phi \) for a positive constant \( c \). The goal is to provide a higher-order approximation of the energy functional that improves upon the first-order approach discussed previously. The expansion is given by:
\[
\left\langle e^{c\tau \Delta} \phi, \phi \right\rangle = \left\langle \phi, \phi \right\rangle + \left\langle c\tau \Delta \phi, \phi \right\rangle + \left\langle \frac{c^2\tau^2}{2}\Delta^2 \phi, \phi \right\rangle + \left\langle \sum_{k=3}^{\infty} \frac{(c\tau \Delta)^k}{k!} \phi, \phi \right\rangle.
\]For \( \phi \in H^2(\Dc) \), we can derive
\begin{align}  
	\left\langle e^{\tau \Delta} \phi, \phi \right\rangle &= \left\langle \phi, \phi \right\rangle + \left\langle \tau \Delta \phi, \phi \right\rangle + \left\langle \frac{\tau^2}{2} \Delta^2 \phi, \phi \right\rangle + O(\tau^3), \label{expand_tau} \\
	\left\langle e^{\frac{\tau}{2} \Delta} \phi, \phi \right\rangle &= \left\langle \phi, \phi \right\rangle + \left\langle \frac{\tau}{2} \Delta \phi, \phi \right\rangle + \left\langle \frac{\tau^2}{8} \Delta^2 \phi, \phi \right\rangle + O(\tau^3). \label{expand_tauo2}
\end{align}
A linear combination of \eqref{expand_tau} and \eqref{expand_tauo2} gives
\[
\left\langle e^{\tau \Delta} \phi, \phi \right\rangle - 4 \left\langle e^{\frac{\tau}{2} \Delta}\phi, \phi \right\rangle = -3 \|\phi\|_2^2 - \left\langle \tau \Delta \phi, \phi \right\rangle + O(\tau^3).
\]
For \( \Vert \phi\Vert_2 = 1 \) with periodic boundary conditions on \( \partial\Dc \), it holds that
\begin{equation}\label{app_GinzbOrd2}
	-\int_{\Dc} \frac{1}{2} (\Delta \phi) \bar{\phi} \, d\x = \frac{3}{2\tau} + \frac{1}{2\tau} \int_{\Dc} \left( |e^{\frac{\tau}{2} \Delta} \phi|^2 - 4 |e^{\frac{\tau}{4} \Delta} \phi|^2 \right) d\x + O(\tau^2).
\end{equation}
Utilizing this result, we derive the second-order accurate relaxed energy functional:
\begin{equation} \label{eng_concOrd2}
	E^{2,\tau}(\phi) = \frac{3}{2\tau} + \int_{\Dc} \left[ \frac{1}{2\tau} \left( \left| e^{\frac{\tau}{2} \Delta} \phi \right|^2 - 4 \left| e^{\frac{\tau}{4} \Delta} \phi \right|^2 \right) + V(\x) |\phi|^2 + \frac{\beta}{2} |\phi|^4  \right] \, d\x.
\end{equation}
Accordingly, the associated optimization problem is defined as
\begin{equation} \label{prob_relxOrd2}
	\min_{\Vert \phi\Vert_2=1} E^{2,\tau}(\phi).
\end{equation}
As shown in \eqref{app_GinzbOrd2}, it holds that
\[
\lim_{\tau \to 0} E^{2,\tau}(\phi) = E(\phi), \quad \text{for } \phi \in H^2(\Dc),
\]
thereby demonstrating the consistency between the relaxed optimization problem \eqref{prob_relxOrd2} and the original problem \eqref{prob_orig}. 

While this second-order relaxation improves accuracy, the analysis of the existence of local minimizers of \eqref{prob_relxOrd2} remains a nontrivial challenge within the current theoretical framework and is left for future investigation. In this subsection, we focus on the development of an efficient numerical algorithm. To this end, the quartic nonlinearity \( |\phi|^4 \) in \eqref{eng_concOrd2} is truncated using the function \( \widetilde{F}(\phi) \) defined in \eqref{trunc_nonlin} (similar arguments could show the minimizer satisfies similar $L^\infty$ bounds with probably different $M$), and a regularization parameter \( \kappa > 0 \) is introduced. The resulting truncated energy functional is given by
\begin{equation}\label{eng_truncOrd2}
	\widetilde{E}^{2,\tau}(\phi) \coloneqq \frac{3}{2\tau} + \int_{\mathcal{D}} 
	\frac{1}{2\tau} \left( \left| e^{\frac{\tau}{2} \Delta} \phi \right|^2 - 4 \left| e^{\frac{\tau}{4} \Delta} \phi \right|^2 \right) 
	+ V(\x) |\phi|^2 + \frac{\beta}{2} \widetilde{F}(\phi) - \kappa |\phi|^2\, d\x + \kappa,
\end{equation}
and $ \widetilde{E}^{2,\tau}$ would produce the same minimizer as that of $ E^{2,\tau}$. Hence, it suffices to consider $\widetilde{E}^{2,\tau}$ hereafter.

\begin{lemma}\label{lem:concavityOrd2}
	For \(\tau > 0\) and \(\kappa \ge \Vert V(\x)\Vert_\infty + 3\beta M^2\) sufficiently large, under Assumption~\ref{assum:BEC}, the functional \(\widetilde{E}^{2,\tau}(\phi)\) defined in \eqref{eng_truncOrd2} is concave on the set \(\{\phi \mid \Vert \phi \Vert_2 \le 1\}\).
\end{lemma}

\begin{proof}
	The second variation of \(\widetilde{E}^{2,\tau}(\phi)\) evaluated at \(\phi\) for a perturbation \(\eta\) takes the form
	\begin{align*}
		\nabla^2 \widetilde{E}^{2,\tau}(\phi)&[\eta, \eta] = \real\left[\int_\Dc \left(\frac{1}{\tau} \left( \left| e^{\frac{\tau}{2} \Delta} \eta \right|^2 - 4 \left| e^{\frac{\tau}{4} \Delta} \eta \right|^2 \right) - 2(\kappa - V(\x)) \odot (\eta \bar{\eta}) \right)\,d\x\right. \\[0.5em]
		& \left.+ \int_{|\phi(\x)| \le M} \left( 4\beta|\phi|^2 \odot (\eta \bar{\eta}) + 2\beta \phi^2 \odot (\bar{\eta} \bar{\eta}) \right) d\x - \int_{|\phi(\x)| \ge M} 6\beta M^2|\eta|^2 \, d\x\right].
	\end{align*}
	The spectral decomposition yields
	\[
	\int_\Dc \frac{1}{\tau} \left( \left| e^{\frac{\tau}{2} \Delta} \eta \right|^2 - 4 \left| e^{\frac{\tau}{4} \Delta} \eta \right|^2 \right)\, d\x \leq 0.
	\]
	Moreover, for \(\kappa \ge \Vert V(\x)\Vert_\infty + 3\beta M^2\), we obtain
	\begin{align*}
		- 2\int_\Dc \left(\kappa - V(\x)\right) \odot (\eta \bar{\eta}) \,d\x 
		&+ \int_{|\phi(\x)| \le M} \left( 4\beta|\phi|^2 \odot (\eta \bar{\eta}) + 2\beta \phi^2 \odot (\bar{\eta} \bar{\eta}) \right) \,d\x \\
	&	- \int_{|\phi(\x)| \ge M} 6\beta M^2|\eta|^2 \, d\x \leq 0.
	\end{align*}
	Hence, \(\widetilde{E}^{2,\tau}(\phi)\) is concave on the set \(\{\phi \mid \Vert \phi \Vert_2 \le 1\}\) with  \(\kappa \ge \Vert V(\x)\Vert_\infty + 3\beta M^2\). This completes the proof.
\end{proof}

The concavity property of \(\widetilde{E}^{2,\tau}(\phi)\) guarantees that any local minimizer of the constrained optimization problem
$$
\min_{\|\phi\|_2 \leq 1} \widetilde{E}^{2,\tau}(\phi)  
$$  
must occur on the unit sphere \(\{\phi \in L^2(\mathcal{D}) \mid \|\phi\|_2 = 1\}\). Therefore, we have 
$$
\min_{\|\phi\|_2 \leq 1} \widetilde{E}^{2,\tau}(\phi) = \min_{\|\phi\|_2 = 1} \widetilde{E}^{2,\tau}(\phi).  
$$  
This equivalence allows us to solve \eqref{prob_relxOrd2} using sequential linear programming.  At each iteration, we  solve the subproblem
\begin{equation}\label{prob_SLPA2}
	\phi^{n+1} = \arg\min_{\Vert \phi\Vert_2 \leq 1} \widetilde{E}^{2,\tau}(\phi^n) + \langle \nabla \widetilde{E}^{2,\tau}(\phi^n), \phi - \phi^n \rangle.
\end{equation}
The solution to \eqref{prob_SLPA2} can be derived using an approach similar to that used in Lemma~\ref{lem:EconcOrd1solSLPA}. 
\begin{lemma}\label{lem:solSLPA2}
	For \(\tau > 0\),   the solution of \eqref{prob_SLPA2} is
	\begin{equation}\label{num_sol2}
		\phi^{n+1} = \frac{\frac{1}{\tau}\left(-\e^{\tau \Delta} \phi^n+4\e^{\frac{\tau}{2}\Delta} \phi^n\right) - 2V(\x)\phi^n-\frac{\beta}{2}  \widetilde{f}(\phi^n) + 2\kappa \phi^n}{\left\Vert \frac{1}{\tau}\left(-\e^{\tau \Delta} \phi^n+4\e^{\frac{\tau}{2}\Delta} \phi^n\right) - 2V(\x)\phi^n -\frac{\beta}{2}  \widetilde{f}(\phi^n) + 2\kappa \phi^n \right\Vert_2},
	\end{equation}
	where $\widetilde{f}(\cdot)$ is given in \eqref{trunc_nonlinDerv}.
\end{lemma}
The stability of the relaxed energy functional under the iterative scheme can also be established. The techniques used for this analysis are similar to those in Theorem~\ref{thm:EconcOrd1_disp}, and the detailed proof is omitted here for brevity. The result is stated as follows:
\begin{theorem}\label{thm:EconcOrd2_disp}
	For \(\tau > 0\),  \(\kappa\ge\Vert V(\x)\Vert_\infty+3\beta M^2 \),  the numerical energy functional updated via \eqref{num_sol2} satisfies
	\begin{equation*}
		\widetilde{E}^{2,\tau}(\phi^{n+1}) \le \widetilde{E}^{2,\tau}(\phi^n),\quad n\ge 0,
	\end{equation*}
	where $ \widetilde{E}^{2,\tau}(\cdot)$ is the relaxed energy functional as given in \eqref{eng_truncOrd2}.
\end{theorem}

Theorem~\ref{thm:EconcOrd2_disp} demonstrates that the iterative scheme ensures non-increasing energy values, preserving stability at each iteration. The pseudocode is presented in Algorithm~\ref{Alg:EconcOrd2}.
\begin{algorithm}[H]
	\caption{Algorithm for computing the local minimum of \eqref{prob_relxOrd2}} \label{Alg:EconcOrd2}
	\begin{algorithmic}
		\State {\bf Input:} Let $\Dc$ be a given bounded domain, $\tau > 0$, $\tol > 0$, $V(\x)$,  $\kappa > 0$, $N_{\max}$ be the maximum number of outer iterations, and $\phi^0$ be the initial condition.
		\State {\bf Output:} $\phi^{n+1}$, the numerical approximation of the local minimum of \eqref{prob_relxOrd2}.
		
		\State Initialize $n = 0$.
		
		\State Update $\phi^{n+1}$ by \eqref{num_sol2}.
		
		\If{the  stopping criterion is met}
		\State Terminate the iteration.
		\EndIf
		
		\State Set $n = n + 1$.
		\State \textbf{End While}
	\end{algorithmic}
\end{algorithm}
\subsection{Adaptive algorithms}
The results in Sections~\ref{sec:algOrd1} and \ref{sec:algOrd2} indicate that the numerical solutions obtained by Algorithms~\ref{Alg:EconcOrd1} and \ref{Alg:EconcOrd2} will converge to the ground state of the original problem~\eqref{prob_orig} as \( \tau \to 0 \). To enhance the efficiency of these algorithms when using a small value of \( \tau = \tau_f \), we first compute a solution with a relatively large step size \( \tau_0 \). This solution is then used as the initial value for subsequent iterations with smaller step sizes. We continue to reduce the step size in this manner until \( \tau = \tau_f \). The details of the implementation are described in Algorithms~\ref{Alg:EconcOrd1_Adap} and~\ref{Alg:EconcOrd2_Adap}.
\begin{algorithm}[H]
	\caption{Adaptive $\tau$ algorithm for Algorithm~\ref{Alg:EconcOrd1}.} \label{Alg:EconcOrd1_Adap}
	\begin{algorithmic}
		\State {\bf Input:} $\Dc$: bounded domain; $\tol$: tolerance; $V(\x)$: potential function; $\kappa$: non-negative parameter; $N_{\max}$: maximum outer iterations; $\tau_0$: the coarsest step size; $\tau_f$: lower bound for time step size; $r$: reduction factor for step size; $\phi^0$: initial condition.
		\State {\bf Output:} $\phi^{n+1}$, the numerical approximation of the local minimum of \eqref{prob_relxOrd1}.
		\While {$\tau \ge \tau_f$}
		\State Run Algorithm~\ref{Alg:EconcOrd1}.
		\State Set $\tau = \tau / r$.
		\EndWhile
	\end{algorithmic}
\end{algorithm}
\begin{algorithm}[H]
	\caption{Adaptive  $\tau$ algorithm for Algorithm~\ref{Alg:EconcOrd2}.} \label{Alg:EconcOrd2_Adap}
	\begin{algorithmic}
		\State {\bf Input:} $\Dc$: bounded domain; $\tol$: tolerance; $V(\x)$: potential function; $\kappa$: non-negative parameter; $N_{\max}$: maximum outer iterations; $\tau_0$: the coarsest step size; $\tau_f$: lower bound for time step size; $r$: reduction factor for step size; $\phi^0$: initial condition.
		\State {\bf Output:} $\phi^{n+1}$, the numerical approximation of the local minimum of \eqref{prob_relxOrd2}.
		\State Set $\tau=\tau_0.$
		\While {$\tau \ge \tau_f$}
		\State Run Algorithm~\ref{Alg:EconcOrd2}.
		\State Set $\tau = \tau / r$.
		\EndWhile
	\end{algorithmic}
\end{algorithm}

\section{Numerical results} \label{sec:numerical_results}
In this section, we demonstrate the stability, convergence, efficiency, and energy dissipation properties of Algorithms~\ref{Alg:EconcOrd1} to~\ref{Alg:EconcOrd2_Adap} through a series of numerical experiments. Spatial discretization is carried out using the Fourier pseudospectral method. The maximum number of iterations is set to \( N_{\max} = 80000 \). In the numerical results, \( E \) represents the numerical approximation of the original energy functional \eqref{eng_orig}, while \( E^\tau \) denotes the relaxed energy functional. The stopping criterion is defined as  
\[
\frac{\Vert \phi^{n+1} - \phi^n \Vert_\infty}{\tau} \leq \tol,
\]  
where \(\tol\) is a prescribed tolerance. All algorithms are implemented in MATLAB and tested on a laptop with a 2.7 GHz Intel Core i5 processor and 8 GB of RAM.

\subsection{Numerical results in 1D}
\begin{example}\label{Ex_omg01D}
	In this example, we consider the local minimum of   \eqref{eng_concOrd1} with  periodic boundary conditions on the interval $[-16, 16]$. The potential function and parameters are specified as
	\begin{align}
		V(x) = \frac{x^2}{2} + 25 \sin^2\left(\frac{\pi x}{4}\right), \quad \beta = 250, \quad \phi_0(x) = \exp\left(-\frac{|x|^2}{2}\right)/\pi^{1/4}.
	\end{align}
\end{example}
\subsubsection{\texorpdfstring{Influence of parameter $\kappa$ on algorithm performance}{Influence of parameter kappa on algorithm performance}}

In Theorems~\ref{thm:EconcOrd1_disp} and \ref{thm:EconcOrd2_disp}, it is shown that choosing \(\kappa\) sufficiently large guarantees the numerical dissipation of the relaxed energy for Algorithms~\ref{Alg:EconcOrd1}--\ref{Alg:EconcOrd2_Adap}.   Here,   \(\x\) denotes the higher-dimensional case, while \(x\) represents the one-dimensional  scenario. To investigate its impact, we tested three different choices of \(\kappa\) in one-dimensional examples: two cases where \(\kappa\) varies with the iteration number, denoted as \(\kappa^n\), and one case where \(\kappa\) remains fixed throughout the iterations. In the numerical experiments, we uniformly refer to \(\kappa\) as \(\kappa^n\) to examine the corresponding energy dissipation and computational efficiency.

Figure~\ref{fig:kapModE} displays the numerical relaxed energy for different values of \(\kappa\). The results demonstrate that the proposed algorithms achieve relaxed energy dissipation when an appropriate \(\kappa\) is selected. However, computational efficiency is influenced by the magnitude of \(\kappa\), with larger values requiring more iterations to converge. Additionally, Figure~\ref{fig:kapOrigE} illustrates the decay behavior of the original energy for the proposed algorithms. For Algorithms~\ref{Alg:EconcOrd1}--\ref{Alg:EconcOrd2_Adap}, selecting \(\kappa^n = \| V(x) + \beta |\phi^n|^2 \|_\infty\) fails to ensure monotonic decay of the original energy. In contrast, setting \(\kappa^n = \| V(x) + 3\beta |\phi^n|^2 \|_\infty\) ensures consistent decreases in both the relaxed and original energy across all four algorithms, while also achieving higher computational efficiency compared to using a larger fixed \(\kappa\). Based on these findings, we adopt \(\kappa^n = \| V(\x) + 3\beta |\phi^n|^2 \|_\infty\) for subsequent computations.
\begin{figure}[htbp]
	\centering
	\begin{subfigure}{0.265\textwidth}
		\includegraphics[width=\textwidth,height=0.14\textheight]{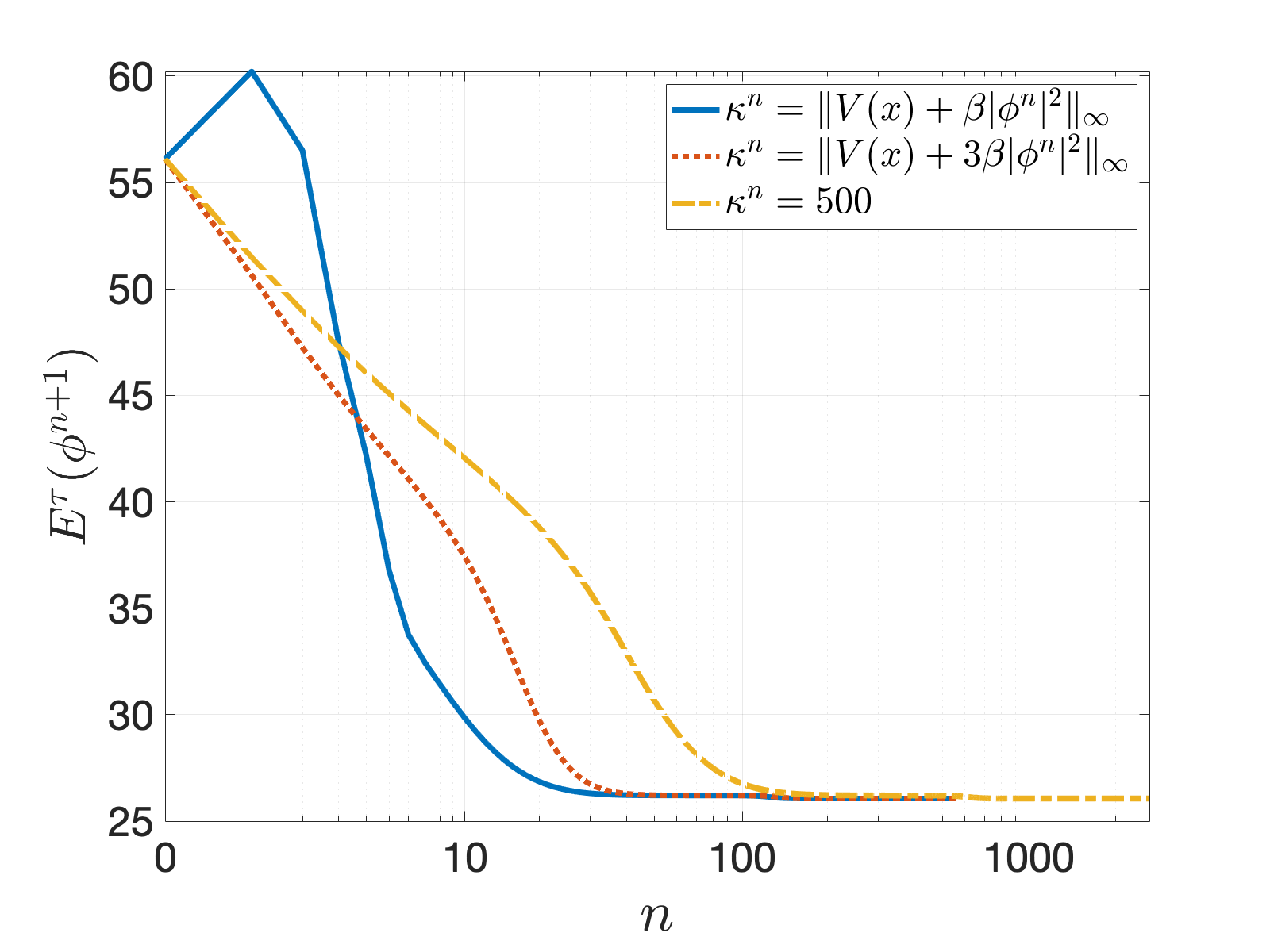}
	\end{subfigure}
    \hspace{-0.5cm} 
	\begin{subfigure}{0.265\textwidth}
		\includegraphics[width=\textwidth,height=0.14\textheight]{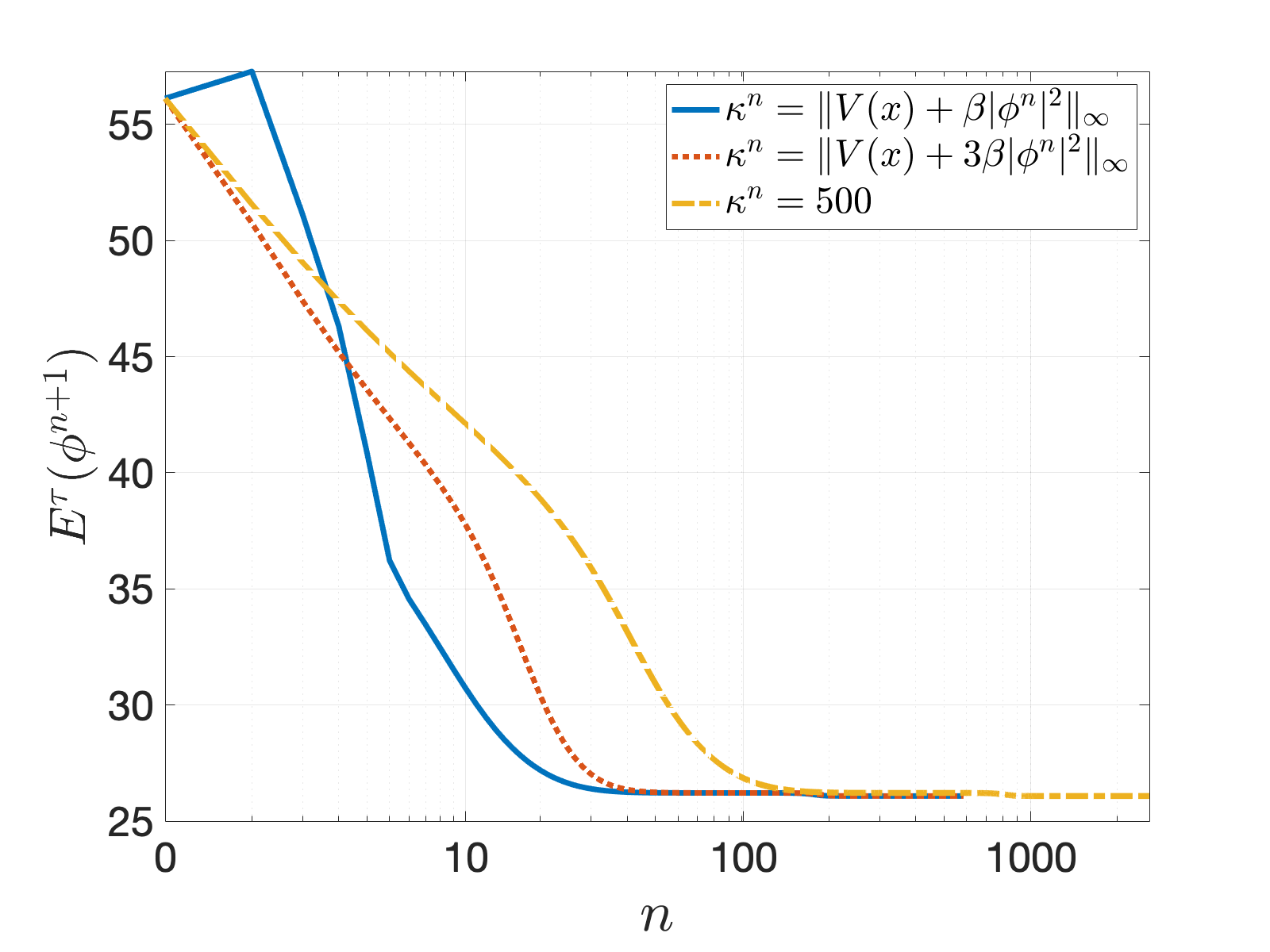}
	\end{subfigure}
    \hspace{-0.5cm} 
	\begin{subfigure}{0.265\textwidth}
		\includegraphics[width=\textwidth,height=0.14\textheight]{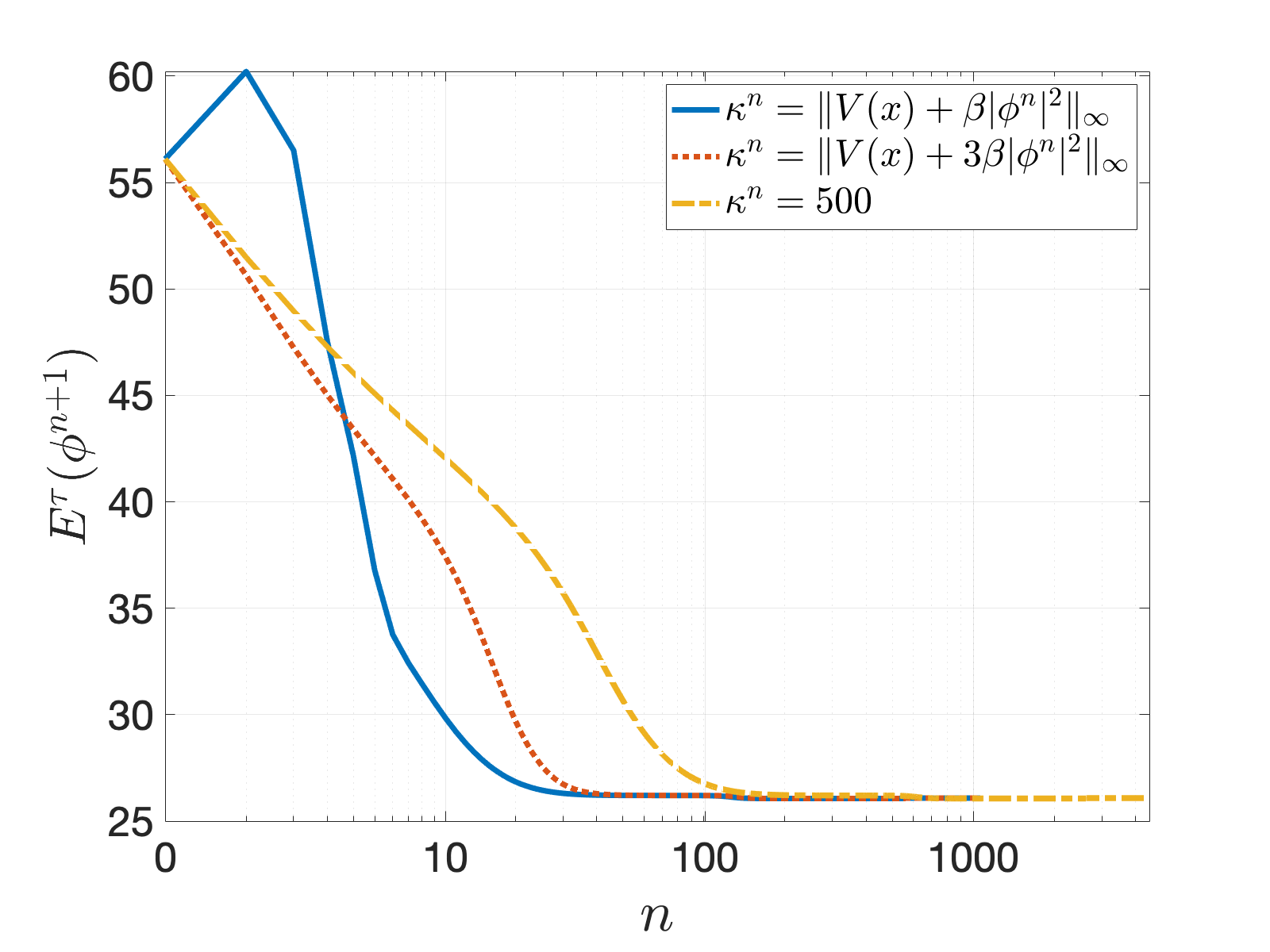}
	\end{subfigure}
    \hspace{-0.5cm} 
	\begin{subfigure}{0.265\textwidth}
		\includegraphics[width=\textwidth,height=0.14\textheight]{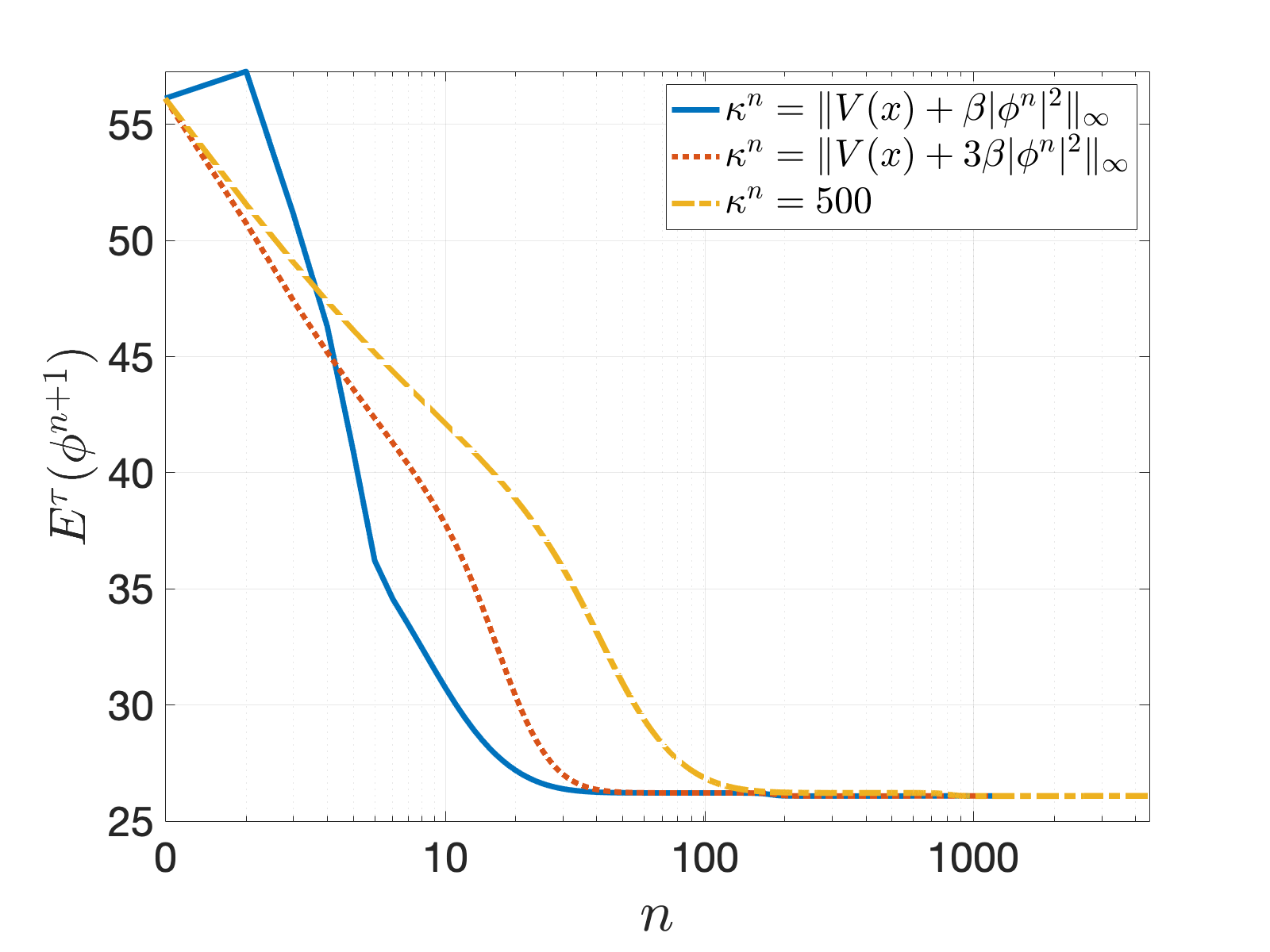}
	\end{subfigure}
    \captionsetup{skip=2pt} 
	\caption{Relaxed energy change per iteration for Example~\ref{Ex_omg01D} with $\tau = 1/8$ $\tau_0=1/8$, $\tau_f=1/80$, $r=10$,   \(h = 1/128\) and \(\tol = 10^{-12}\). From left to right: Algorithm~\ref{Alg:EconcOrd1}, Algorithm~\ref{Alg:EconcOrd2}, Algorithm~\ref{Alg:EconcOrd1_Adap}, and Algorithm~\ref{Alg:EconcOrd2_Adap}.}
	\label{fig:kapModE}
\end{figure}

\begin{figure}[htbp]
	\centering
	\begin{subfigure}{0.265\textwidth}
		\includegraphics[width=\textwidth,height=0.14\textheight]{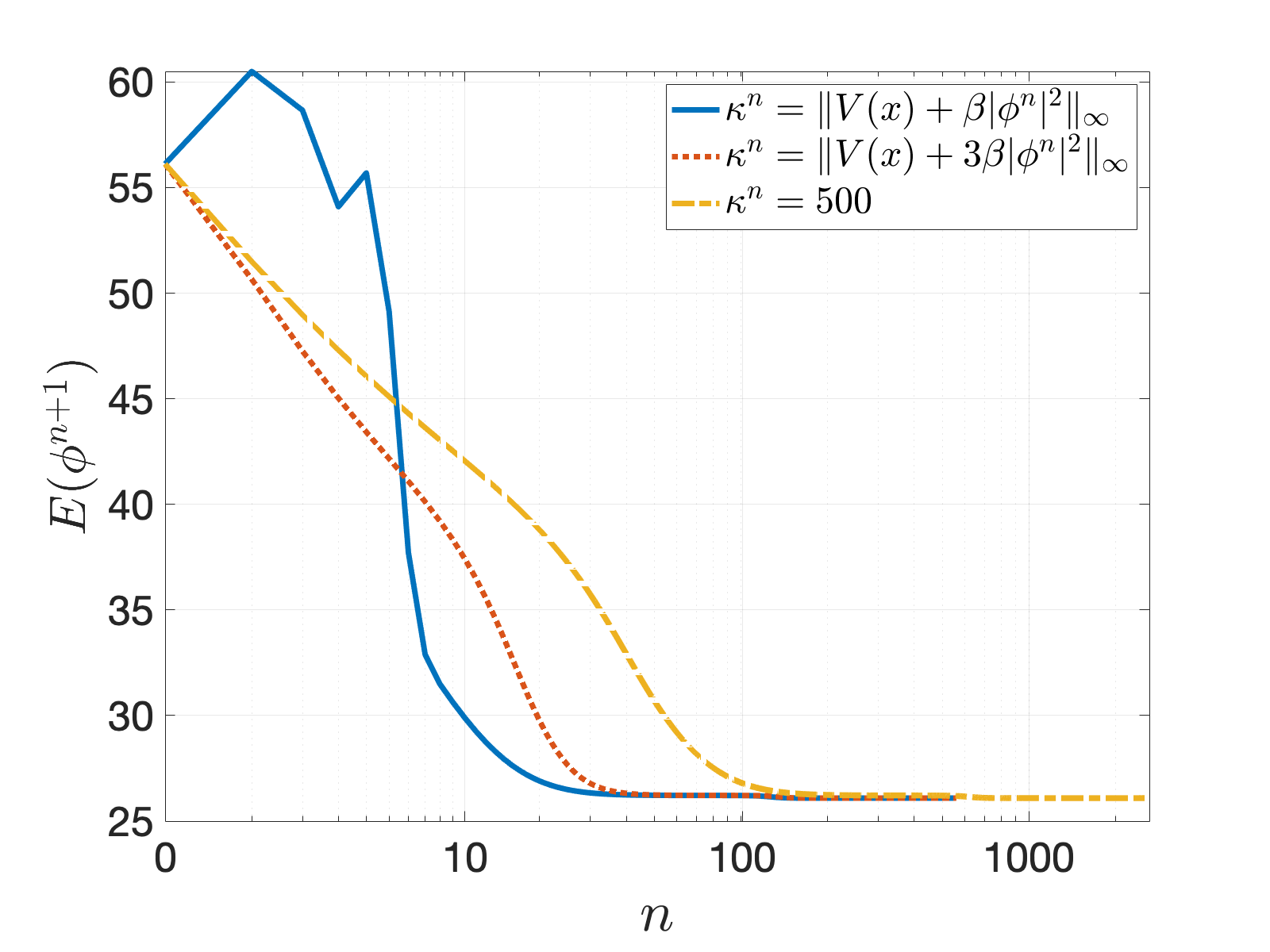}
	\end{subfigure}
 \hspace{-0.5cm} 
	\begin{subfigure}{0.265\textwidth}
		\includegraphics[width=\textwidth,height=0.14\textheight]{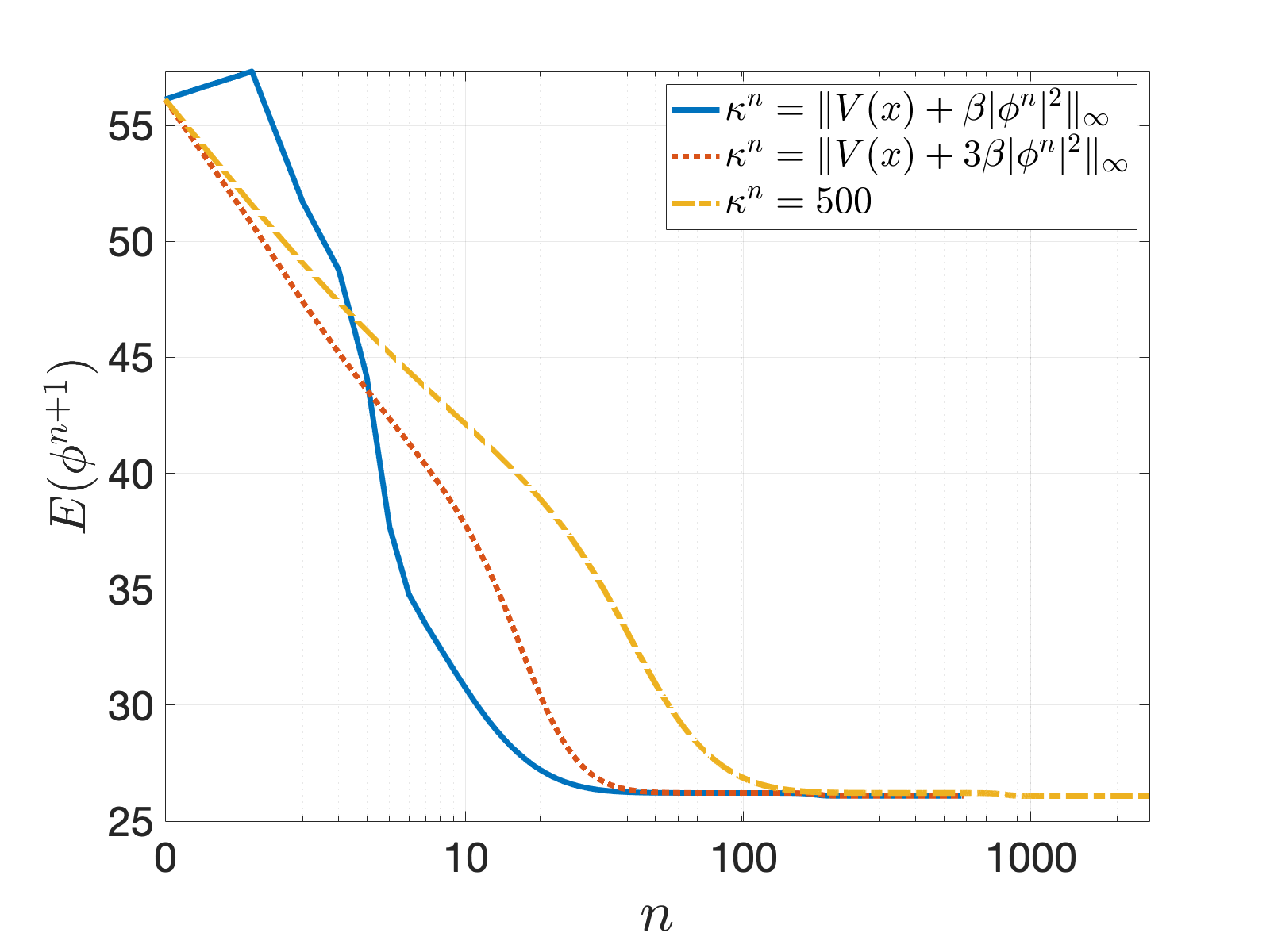}
	\end{subfigure}
 \hspace{-0.5cm} 
	\begin{subfigure}{0.265\textwidth}
		\includegraphics[width=\textwidth,height=0.14\textheight]{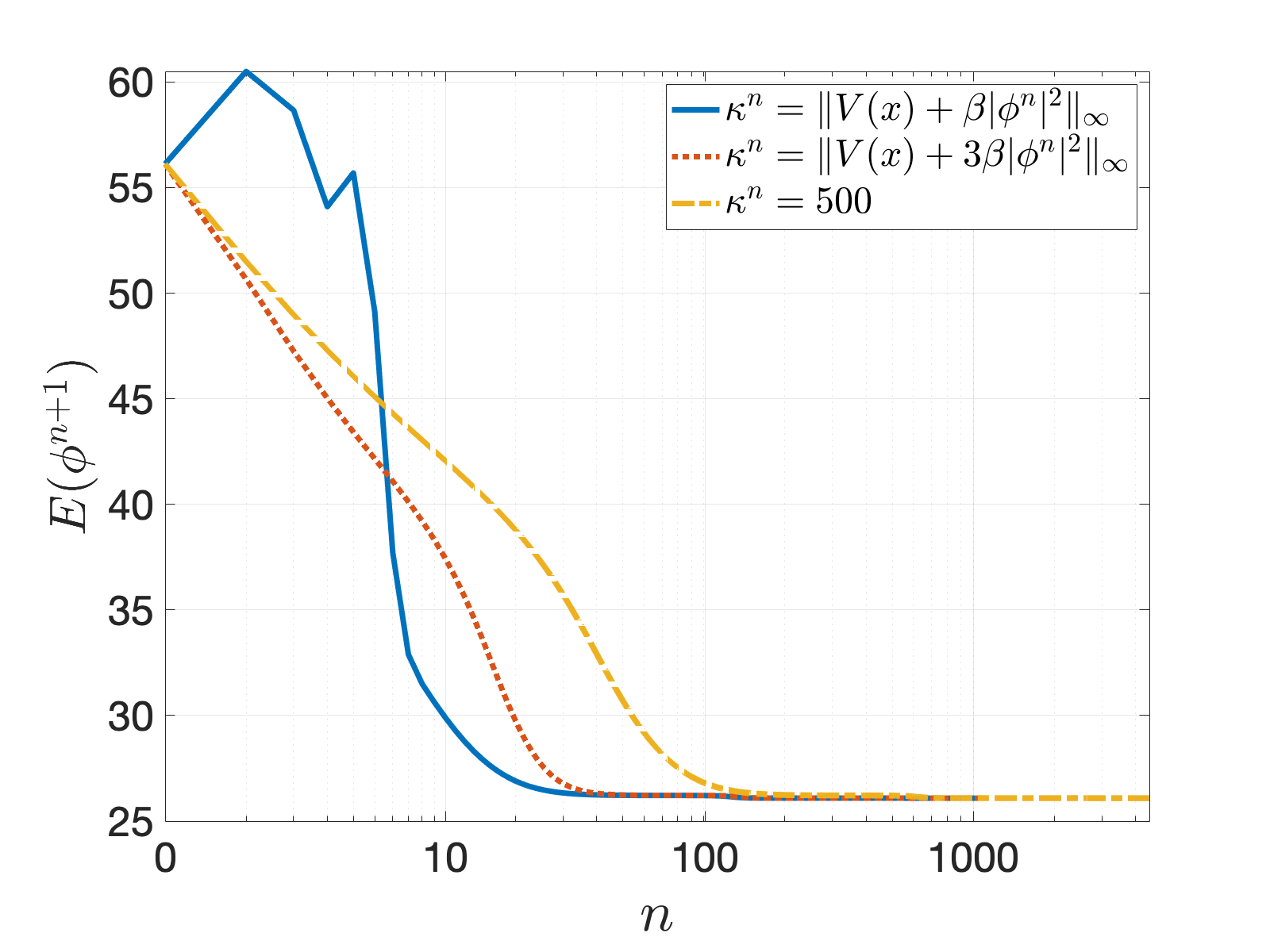}
	\end{subfigure}
 \hspace{-0.5cm} 
	\begin{subfigure}{0.265\textwidth}
		\includegraphics[width=\textwidth,height=0.14\textheight]{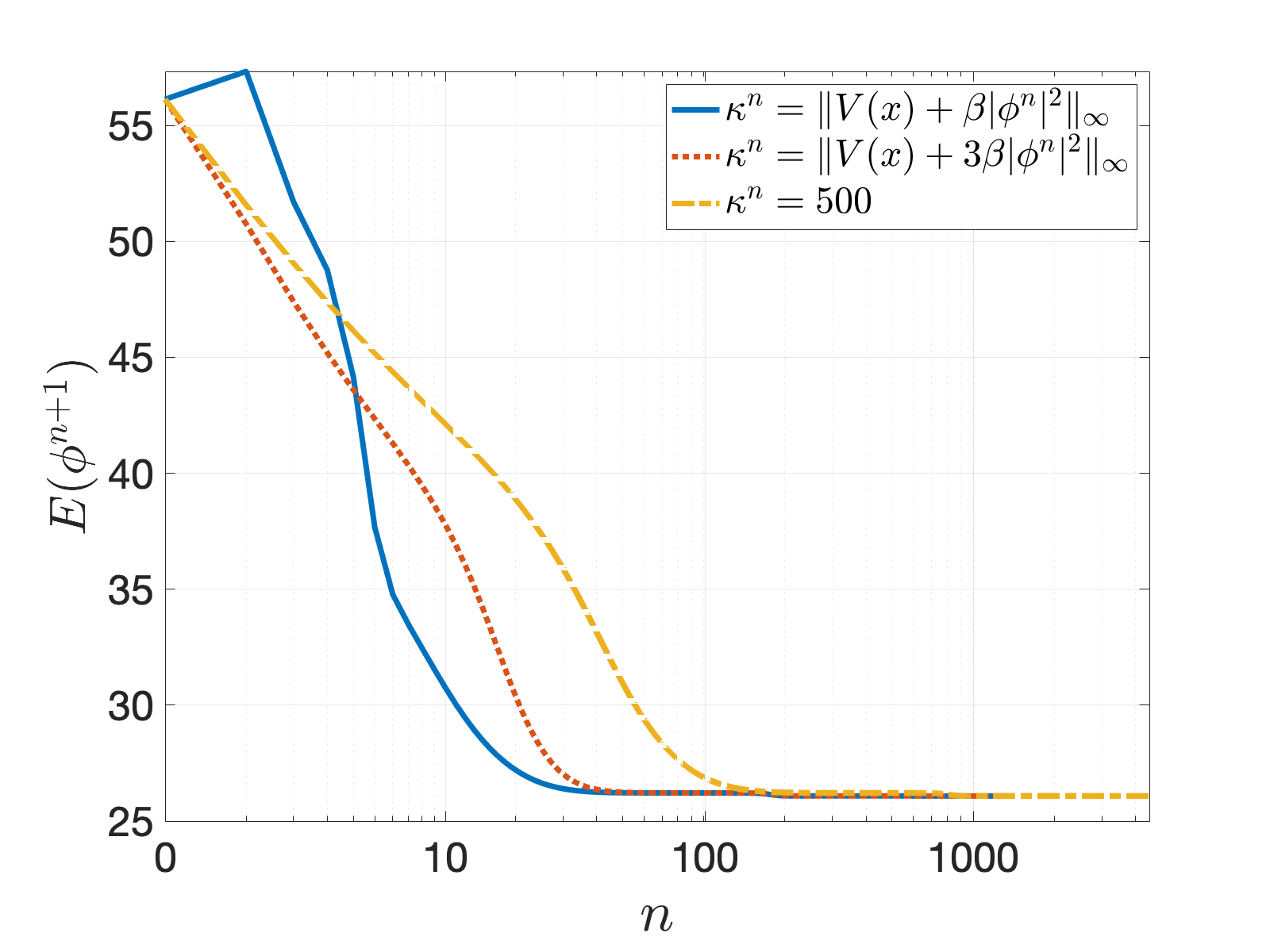}
	\end{subfigure}
    \captionsetup{skip=2pt} 
	\caption{Original  energy change per iteration for Example~\ref{Ex_omg01D} ith $\tau = 1/8$ $\tau_0=1/8$, $\tau_f=1/80$, $r=10$,   \(h = 1/128\) and \(\tol = 10^{-12}\). From left to right: Algorithm~\ref{Alg:EconcOrd1}, Algorithm~\ref{Alg:EconcOrd2}, Algorithm~\ref{Alg:EconcOrd1_Adap}, and Algorithm~\ref{Alg:EconcOrd2_Adap}.}
	\label{fig:kapOrigE}
\end{figure}

\subsubsection{Accuracy test}
To test the  accuracy of the numerical scheme in this paper, we first compute a reference solution $\phi_g^{\text{ref}}$ using the GPELAB toolbox   \cite{AntDub15gpelab} on a very fine mesh with $h=1/128$, $\tau=1/1000$, $\tol=10^{-12}$. We denote $\phi_g$ as the numerical local minimum obtained on coarser meshes. Additionally, we define
\[E_g^{\text{ref}} := E(\phi_g^{\text{ref}}), \quad \mu_g^{\text{ref}} := \mu(\phi_g^{\text{ref}}), \quad E_g := E(\phi_g), \quad \mu_g := \mu(\phi_g).\]

In Tables~\ref{tab:TempConv1DOmg0} and \ref{tab:TempConv1DOmg0_Ord2}, we investigate the convergence  of Algorithms~\ref{Alg:EconcOrd1} and \ref{Alg:EconcOrd2} with respect to \( \tau \).  The results indicate that the numerical solution \( \phi_g \) and the corresponding chemical potential \( \mu(\phi_g) \) obtained by Algorithm~\ref{Alg:EconcOrd1} converge linearly with respect to \( \tau \) to the local minimum and chemical potential of the original problem \eqref{prob_orig}, while the numerical solution \( \phi_g \) and the chemical potential \( \mu(\phi_g) \) derived from Algorithm~\ref{Alg:EconcOrd2} exhibit convergence of \( O(\tau^2) \).
And the convergence rates of the numerical energy for Algorithm~\ref{Alg:EconcOrd1} and Algorithm~\ref{Alg:EconcOrd2} are \( O(\tau^2) \) and \( O(\tau^4) \), respectively.

To test the spatial accuracy of the proposed algorithms, we implement Algorithm~\ref{Alg:EconcOrd1_Adap} and Algorithm~\ref{Alg:EconcOrd2_Adap} with a fixed \( \tau_f \). The results presented in Tables~\ref{tab:SpConv1DOmg0} and~\ref{tab:SpConv1DOmg0_Ord2} demonstrate that the algorithms converge rapidly by adopting the pseudospectral method in space.
\begin{table}[htbp]
	\centering
	\renewcommand{\arraystretch}{1.2} 
	\setlength{\tabcolsep}{8pt} 
	\begin{tabular}{ccccccc}
		\hline
		\noalign{\vskip 2mm} 
		$\tau$ & $\max\left\vert \phi_g - \phi^{\text{ref}}_g \right\vert$ & Rate & $\left\vert E_g - E^{\text{ref}}_g \right\vert$ & Rate & $\left\vert \mu_g - \mu^{\text{ref}}_g \right\vert$ & Rate \\
		\noalign{\vskip 2mm} \hline
		1/20 & 5.20e-03 & -- & 6.17e-04 & -- & 6.30e-03 & -- \\                                                                               
		1/40 & 2.56e-03 & 1.02 & 1.54e-04 & 2.00 & 3.01e-03 & 1.06 \\                                                                               
		1/80 & 1.26e-03 & 1.03 & 3.79e-05 & 2.02 & 1.47e-03 & 1.04 \\                                                                                
		1/160 & 6.21e-04 & 1.02 & 9.38e-06 & 2.02 & 7.24e-04 & 1.02 \\                                                                                
		1/320 & 3.09e-04 & 1.01 & 2.33e-06 & 2.01 & 3.60e-04 & 1.01 \\                          
		\hline    
	\end{tabular}
	\caption{Numerical results  of Algorithm~\ref{Alg:EconcOrd1} with $h=1/128$ and $\tol=10^{-12}$ for  Example~\ref{Ex_omg01D}.}
	\label{tab:TempConv1DOmg0}
\end{table}
\begin{table}[htbp]
	\centering
	\renewcommand{\arraystretch}{1.2} 
	\setlength{\tabcolsep}{8pt} 
	\begin{tabular}{ccccccc}
		\hline
		\noalign{\vskip 2mm} 
		$\tau$ & $\max\left\vert \phi_g - \phi^{\text{ref}}_g \right\vert$ & Rate & $\left\vert E_g - E^{\text{ref}}_g \right\vert$ & Rate & $\left\vert \mu_g - \mu^{\text{ref}}_g \right\vert$ & Rate \\
		\noalign{\vskip 2mm} \hline
		1/20 & 6.72e-04 & -- & 1.14e-05 & -- & 4.60e-04 & --  \\                                                                                  
		1/40 & 2.05e-04 & 1.72 & 1.10e-06 & 3.37 & 1.27e-04 & 1.85 \\                                                                                 
		1/80 & 5.83e-05 & 1.81 & 9.17e-08 & 3.59 & 3.36e-05 & 1.92  \\       
		
		1/160 & 1.58e-05 & 1.88 & 6.82e-09 & 3.75 & 8.63e-06 & 1.96  \\      
		
		1/320 & 4.13e-06 & 1.94 & 4.71e-10 & 3.86 & 2.18e-06 & 1.99  \\                                                                             
		\hline    
	\end{tabular}
	\caption{Numerical results  of Algorithm~\ref{Alg:EconcOrd2} with $h=1/128$ and $\tol=10^{-12}$ for  Example~\ref{Ex_omg01D}.}
	\label{tab:TempConv1DOmg0_Ord2}
\end{table}
\begin{table}[htbp]
	\centering
	\renewcommand{\arraystretch}{1.2} 
	\setlength{\tabcolsep}{8pt} 
	\begin{tabular}{ccccccc}
		\hline
		\noalign{\vskip 2mm} 
		$h$ & $\max\left\vert \phi_g - \phi^{\text{ref}}_g \right\vert$ & Rate & $\left\vert E_g - E^{\text{ref}}_g \right\vert$ & Rate & $\left\vert \mu_g - \mu^{\text{ref}}_g \right\vert$ & Rate \\
		\noalign{\vskip 2mm} \hline
		1 & 7.93e-03 & -- & 4.44e-04 & -- & 9.77e-02 & --\\                     
		1/2 & 1.20e-03 & 2.72 & 1.95e-04 & 1.19 & 4.11e-03 & 4.57 \\
		1/4 & 2.20e-06 & 9.10 & 4.98e-08 & 11.93 & 5.03e-07 & 13.00 \\                   
		1/8 & 2.11e-08 & 6.71 & 7.11e-15 & 22.74 & 5.83e-08 & 3.11 \\                       
		\hline    
	\end{tabular}
	\caption{Spatial resolution of Algorithm~\ref{Alg:EconcOrd1_Adap} with $\tau_0=10^{-1}$, $\tau_f=10^{-6}$, $r=10$ and $\tol=10^{-12}$  for Example~\ref{Ex_omg01D}.}\label{tab:SpConv1DOmg0}
\end{table}
\begin{table}[htbp]
	\centering
	\renewcommand{\arraystretch}{1.2} 
	\setlength{\tabcolsep}{8pt} 
	\begin{tabular}{ccccccc}
		\hline
		\noalign{\vskip 2mm} 
		$h$ & $\max\left\vert \phi_g - \phi^{\text{ref}}_g \right\vert$ & Rate & $\left\vert E_g - E^{\text{ref}}_g \right\vert$ & Rate & $\left\vert \mu_g - \mu^{\text{ref}}_g \right\vert$ & Rate \\
		\noalign{\vskip 2mm} \hline
		1 & 7.93e-03 & -- & 4.44e-04 & -- & 9.77e-02 & --\\                     
		1/2 & 1.20e-03 & 2.72 & 1.95e-04 & 1.19 & 4.11e-03 & 4.57  \\
		1/4 & 2.20e-06 & 9.10 & 4.98e-08 & 11.93 & 5.76e-07 & 12.80 \\           1/8 & 2.47e-08 & 6.47 & 7.11e-15 & 22.74 & 1.56e-08 & 5.21\\                   
		\hline    
	\end{tabular}
	\caption{Spatial resolution of Algorithm~\ref{Alg:EconcOrd2_Adap} with $\tau_0=10^{-1}$, $\tau_f=10^{-6}$, $r=10$ and $\tol=10^{-12}$  for  Example~\ref{Ex_omg01D}.}\label{tab:SpConv1DOmg0_Ord2}
\end{table}
\subsection{Numerical results in 2D}
\begin{example}\label{Ex:2Domg0}\cite[Example 4.2]{YinHuaZha}
	In this example, we consider the approximation of the ground state \eqref{eng_concOrd1} with periodic boundary conditions on  a bounded domain $[-8,8]^2$ and choose
	$V(\x)=\frac{1}{2}|\x|^2,\quad \beta=300.$
    The initial condition is 
$\phi(\x)=\frac{\e^{-V(\x)}}{\Vert\e^{-V(\x)}\Vert_2}.$
	In the implementation,  the iteration stops when
	\[\Vert \phi^{n+1}-\phi^{n}\Vert_\infty/\tau\le 10^{-7}.\]
\end{example}
In this example, we primarily aim to demonstrate the efficiency of the proposed algorithms and to verify the relationship between the original energy functional and the relaxed energy functionals \eqref{eng_concOrd1} and \eqref{eng_concOrd2}.

First, we plot the evolution of the approximated energy of \eqref{eng_orig} in Figure~\ref{fig:2dOmg0EevoComp}. From this figure, we observe that the proposed algorithms converge to the neighborhood of the ground state within a few iterations. Next, we compare the difference between the original numerical energy  and the relaxed numerical energy   in the last column of Table~\ref{table:2dOmg0_Eff}, verifying that 
\[
|E(\phi) - E^{1,\tau}(\phi)| = O(\tau), \quad |E(\phi) - E^{2,\tau}(\phi)| = O(\tau^2).
\]
Additionally, the CPU times for executing Algorithms~\ref{Alg:EconcOrd1}--\ref{Alg:EconcOrd2_Adap} are shown in the third column of Table~\ref{table:2dOmg0_Eff}. These results highlight several key points:
\begin{itemize}
	\item Algorithms~\ref{Alg:EconcOrd1} and~\ref{Alg:EconcOrd2} converge to the numerical ground state in a shorter time than Algorithms~\ref{Alg:EconcOrd1_Adap} and~\ref{Alg:EconcOrd2_Adap}.
	\item The CPU time required by Algorithms~\ref{Alg:EconcOrd1_Adap} and~\ref{Alg:EconcOrd2_Adap} increases more gradually with smaller values of \( \tau \) compared to Algorithms~\ref{Alg:EconcOrd1} and~\ref{Alg:EconcOrd2}. As \( \tau \) becomes very small, these adaptive algorithms will be more efficient.
	\item Algorithms~\ref{Alg:EconcOrd2} and~\ref{Alg:EconcOrd2_Adap} achieve higher accuracy than Algorithms~\ref{Alg:EconcOrd1} and~\ref{Alg:EconcOrd1_Adap},  which aligns with the theoretical analysis.
\end{itemize}
\begin{figure}[htbp]
\centering
\includegraphics[
    width=.45\textwidth,
    clip,
    trim = 2cm 0.5cm 0cm 1cm
]{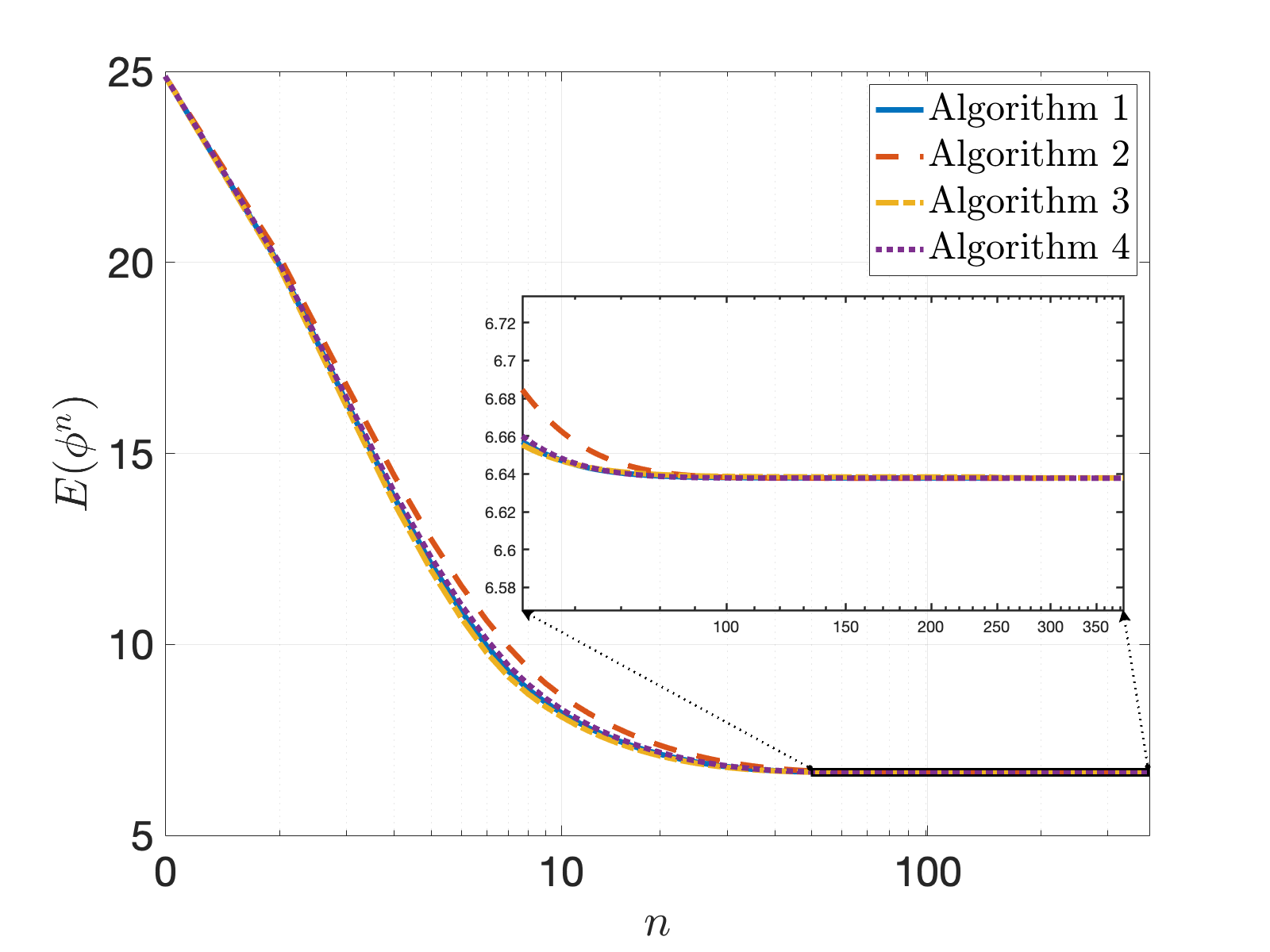}
\captionsetup{skip=2pt} 
\caption{Evolution of the approximated energy of different algorithms with $\tau=1/16$, $r=2$ for Example~\ref{Ex:2Domg0}.}
\label{fig:2dOmg0EevoComp}
\end{figure}
\begin{table}[htbp]
	\centering
	\begin{tabular}{cccccc} 
		\hline
		Algorithm & $\tau$ & CPU(s) & $\Vert \phi^{\text{ref}}_g-\phi_g\Vert_\infty$ &  $E(\phi_g)-E^\tau(\phi_g)$  \\
		\hline
		\multirow{5}{*}{Algorithm~\ref{Alg:EconcOrd1}} 
		& $1/64$  & 0.19 & 2.87e-04 & 7.24e-04 \\
		& $1/128$ & 0.17 & 1.43e-04 & 3.60e-04 \\
		& $1/256$ & 0.28 & 7.13e-05 & 1.80e-04 \\
		& $1/512$ & 0.48 & 3.56e-05 & 8.98e-05 \\
		& $1/1024$ & 0.88 & 1.78e-05 & 4.49e-05 \\
		\hline
		\multirow{5}{*}{Algorithm~\ref{Alg:EconcOrd2}} 
		& $1/64$  & 0.39 & 6.82e-06 & 6.05e-06 \\
		& $1/128$ & 0.59 & 1.81e-06 & 1.54e-06 \\
		& $1/256$ & 1.10 & 4.95e-07 & 3.90e-07 \\
		& $1/512$ & 2.13 & 1.55e-07 & 9.79e-08 \\
		& $1/1024$ & 5.11 & 6.91e-08 & 2.46e-08 \\
		\hline
		\multirow{5}{*}{Algorithm~\ref{Alg:EconcOrd1_Adap}} 
		& $1/64$  & 0.19 & 2.87e-04 & 7.24e-04 \\
		& $1/128$ & 0.26 & 1.43e-04 & 3.60e-04 \\
		& $1/256$ & 0.38 & 7.13e-05 & 1.80e-04 \\
		& $1/512$ & 0.59 & 3.56e-05 & 8.98e-05 \\
		& $1/1024$ & 0.98 & 1.78e-05 & 4.49e-05 \\
		\hline
		\multirow{5}{*}{Algorithm~\ref{Alg:EconcOrd2_Adap}} 
		& $1/64$  & 0.37 & 6.82e-06 & 6.05e-06 \\
		& $1/128$ & 0.52 & 1.81e-06 & 1.54e-06 \\
		& $1/256$ & 0.70 & 4.94e-07 & 3.90e-07 \\
		& $1/512$ & 0.94 & 1.52e-07 & 9.79e-08 \\
		& $1/1024$ & 1.30 & 6.27e-08 & 2.46e-08 \\
		\hline
	\end{tabular}
	\caption{Comparison of different algorithms for Example~\ref{Ex:2Domg0} on a $128 \times 128$ discretized mesh with $r = 2$ and $\tol = 10^{-7}$.}
	\label{table:2dOmg0_Eff}
\end{table}

The algorithms proposed in this paper are robust even with larger $\tau$. The resulting outputs can serve as initial conditions for quasi-Newton methods, such as the Riemannian limited memory Broyden--Fletcher--Goldfarb--Shanno (RLBFGS) solver, which enables rapid convergence to the exact local minimum.
\section{Rotating Bose--Einstein condensate}\label{sec:rot_BECs}
In this section, the framework from Section~\ref{sec:relaxed_functionals} is extended to compute the ground state of the rotating BEC. The corresponding energy functional includes a rotational term and is given by
\begin{equation}\label{eng_origrot}
	E_{rot}(\phi) = \int_{\Dc} \left(\frac{1}{2} |\nabla \phi|^2 + V(\x) |\phi|^2 + \frac{\beta}{2} |\phi|^4 - \Omega \overline{\phi} L_z \phi \right)\, d\x,   
\end{equation}
where \( L_z \phi = -i (x \partial_y - y \partial_x) \phi \) represents the \( z \)-component of the angular momentum \( \x \times (-i \nabla \phi) \). Define 
$
R(\x) := \Omega (y, -x).$
And let 
$
\nabla_R \phi = \nabla \phi + i R^\top \phi,$
with the associated second-order operator given by
\[
\nabla_R^2 \phi = \nabla_R \cdot (\nabla_R \phi) = \Delta \phi + 2i R \cdot \nabla \phi - |R|^2 \phi.
\]
For \( \phi \) satisfying periodic boundary conditions on \(\partial \Dc\), the energy functional can be reformulated as
\begin{equation}
	E_{rot}(\phi) = \int_{\Dc} \left( -\frac{1}{2} (\nabla_R^2 \phi) \overline{\phi} + W(\x) |\phi|^2 + \frac{\beta}{2} |\phi|^4 \right) \, d\x,  
\end{equation}
where 
\begin{equation} \label{Wx}
	W(\x) = V(\x) - \frac{|R(\x)|^2}{2}.
\end{equation}

\subsection{Numerical scheme}
In order to  compute the local minima of the rotating BEC,  we build on the methodology introduced in Section~\ref{sec:relaxed_functionals}, applying it to the relaxed energy functional that incorporates rotational effects. The primary tool for this solution involves the expansion
\[
e^{\tau \nabla_R^2} = \sum_{k=0}^{\infty} \frac{(\tau \nabla_R^2)^k}{k!}.
\]
This leads to
\[
\langle e^{\tau \nabla_R^2} \phi, \phi \rangle = \langle \phi, \phi \rangle + \langle \tau \nabla_R^2 \phi, \phi \rangle + \left\langle \sum_{k=2}^{\infty} \frac{(\tau \nabla_R^2)^k}{k!} \phi, \phi \right\rangle.
\]
The methodology in Section~\ref{sec:relaxed_functionals} extends naturally to the rotating BEC problem using
\[
\langle e^{c \tau \Delta} \phi, \phi \rangle = \langle \phi, \phi \rangle + \langle c \tau \Delta \phi, \phi \rangle + \langle \frac{c^2 \tau^2}{2} \nabla_R^4 \phi, \phi \rangle + \left\langle \sum_{k=3}^{\infty} \frac{(c \tau \Delta)^k}{k!} \phi, \phi \right\rangle,
\]
and
\[
\langle e^{c \tau \nabla_R^2} \phi, \phi \rangle = \langle \phi, \phi \rangle + \langle c \tau \nabla_R^2 \phi, \phi \rangle + \langle \frac{c^2 \tau^2}{2} \nabla_R^4 \phi, \phi \rangle + \left\langle \sum_{k=3}^{\infty} \frac{(c \tau \nabla_R^2)^k}{k!} \phi, \phi \right\rangle.
\]
The relaxed energy functionals are defined as
\begin{equation} \label{eng_rotconcOrd1}
	E_{rot}^{1,\tau}(\phi) = \frac{1}{2\tau} + \int_{\Dc} \left( -\frac{1}{2\tau} \left| e^{\frac{\tau}{2} \nabla_R^2} \phi \right|^2 + W(\x) |\phi|^2 + \frac{\beta}{2} |\phi|^4 - \kappa |\phi|^2 \right) d\x + \kappa,
\end{equation}
and
\begin{equation} \label{eng_rotconcOrd2}
	E_{rot}^{2,\tau}(\phi) = \frac{3}{2\tau} + \int_{\Dc} \frac{1}{2\tau} \left( |e^{\frac{\tau}{2} \nabla_R^2} \phi|^2 - 4 |e^{\frac{\tau}{4} \nabla_R^2} \phi|^2 \right) + W(\x) |\phi|^2 + \frac{\beta}{2} |\phi|^4 - \kappa |\phi|^2 d\x + \kappa,
\end{equation}
subject to \( \Vert \phi \Vert_2 = 1. \) Algorithms in Section~\ref{sec:relaxed_functionals} are extended to compute the constrained local minima of \eqref{eng_rotconcOrd1} and \eqref{eng_rotconcOrd2} using the iterations
\[
\phi^{n+1} = \frac{\frac{1}{\tau} e^{\tau \nabla_R^2} \phi^n - 2\phi^n \left( V(\x) + \beta |\phi^n|^2 - \kappa \right)}{\left\Vert \frac{1}{\tau} e^{\tau \nabla_R^2} \phi^n - 2\phi^n \left( V(\x) + \beta |\phi^n|^2 - \kappa \right) \right\Vert_2},
\]
for the first-order algorithm, and
\[
\phi^{n+1} = \frac{\frac{1}{\tau} \left( -e^{\tau \nabla_R^2} \phi^n + 4 e^{\frac{\tau}{2} \nabla_R^2} \phi^n \right) - 2\phi^n \left( V(\x) + \beta |\phi^n|^2 - \kappa \right)}{\left\Vert \frac{1}{\tau} \left( -e^{\tau \nabla_R^2} \phi^n + 4 e^{\frac{\tau}{2} \nabla_R^2} \phi^n \right) - 2\phi^n \left( V(\x) + \beta |\phi^n|^2 - \kappa \right) \right\Vert_2},
\]
for the second-order algorithm.
\subsection{Numerical results}
To evaluate the effectiveness of the proposed numerical scheme, several tests were performed to examine the performance of the algorithm, which extends those presented in Section~\ref{sec:relaxed_functionals}, in solving the energy functional for the rotating BEC. In the numerical implementation, \( e^{\tau \nabla_R^2} \phi^n \) is computed using the algorithm from \cite{BerCrLi}. The parameter \( \kappa\) is updated at each iteration and is determined by
\[
\kappa = \| V(\x) + 3\beta |\phi^n|^2 \|_\infty.
\]
The algorithm settings and notation used in the figures in Section~\ref{sec:numerical_results} are retained in the numerical implementation for computing the ground states of rotating BECs.
\begin{example}\cite[Pages 18-22]{BerCrLi}\label{Ex:rotAnis}
	The ground state approximation of \eqref{eng_rotconcOrd1} is computed with periodic boundary conditions on \( [-12,12]^2 \). Parameters are chosen as
	\[
	\gamma_x = 1.05, \quad \gamma_y = 0.95, \quad \beta = 1000, \quad \kappa = \left\| V(\x) + 3\beta |\phi^n|^2 \right\|_\infty.
	\]
	The Fourier pseudo-spectral method is applied for spatial discretization with a step size \( h = 24/128 \). The initial condition is 
	\[
	\phi^0(\x) = \frac{(1-\Omega)\phi_a + \Omega\phi_b}{\Vert (1-\Omega)\phi_a + \Omega\phi_b \Vert},
	\]
	where
	\[
	\phi_a = \frac{(\gamma_x \gamma_y)^{1/4}}{\sqrt{\pi}} e^{-V(\x)}, \quad \phi_b = \frac{\gamma_x - \gamma_y i y}{\sqrt{\pi}} e^{-V(\x)}.
	\]
	The maximum number of iterations is set to \( N_{\max} = 80000 \), and the iteration stops when
	\[
	\Vert \phi^{n+1} - \phi^n \Vert_\infty / \tau \leq \tol.
	\]
\end{example}

Figures~\ref{fig:2DOmg05} and \ref{fig:2DOmg09} illustrate the values of \( |\phi|^2 \) obtained by the proposed algorithms for different values of \( \Omega \), with the corresponding energy variations depicted in Figures~\ref{fig:2drotEomg05} and \ref{fig:2drotEomg09}. During the iterative process of solving the ground state problem of the rotating BEC, it is observed that the original energy \( E(\phi^n) \) decreases monotonically, further confirming the reliability of the algorithms presented in this work.

Starting from the result produced by Algorithm~\ref{Alg:EconcOrd2_Adap} in Figure~\ref{fig:2DOmg05}, the RLBFGS algorithm is employed to refine the solution. The iteration proceeds until the \( L^2 \)-norm of the Riemannian gradient of \eqref{eng_origrot} drops below \( 10^{-6} \). As shown in Figure~\ref{fig:Omg05rlbfgs}, the left panel depicts the energy decay during the optimization, while the right panel displays the final state of \( |\phi|^2 \). From this figure, it is clear that the results in Figure~\ref{fig:2DOmg05} closely approximate the true ground state, demonstrating the effectiveness of the proposed algorithms as initial conditions for Newton's method. This further substantiates the reliability of the proposed approach.

\begin{figure}[htbp]
	\centering
	\begin{subfigure}{0.24\textwidth}
		\includegraphics[height=\textwidth,width=\textwidth, clip, trim = 2.5cm 3cm 1cm 1cm]{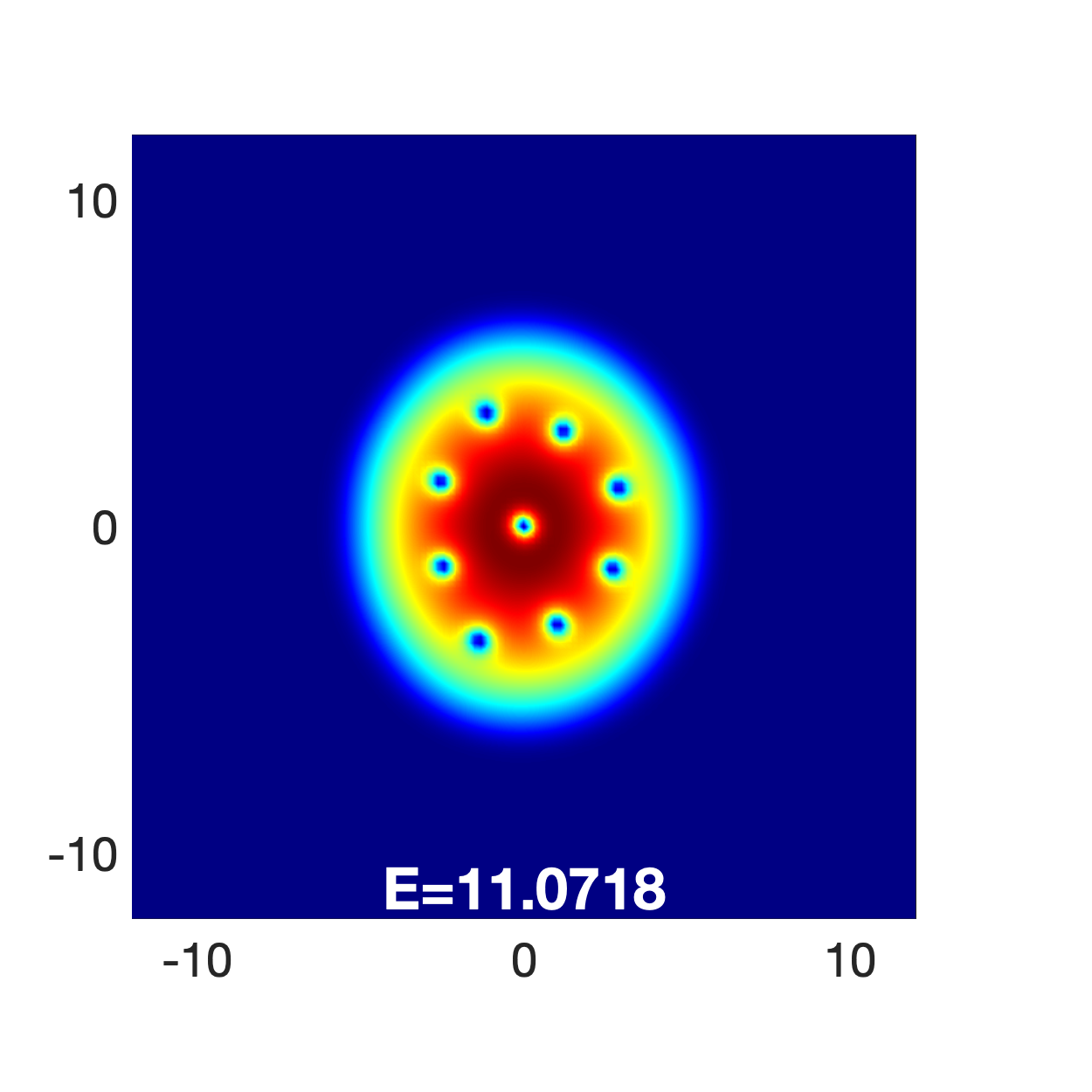}
	\end{subfigure}
	\hspace{-0.5cm} 
	\begin{subfigure}{0.24\textwidth}
		\includegraphics[height=\textwidth,width=\textwidth, clip, trim = 2.5cm 3cm 1cm 1cm]{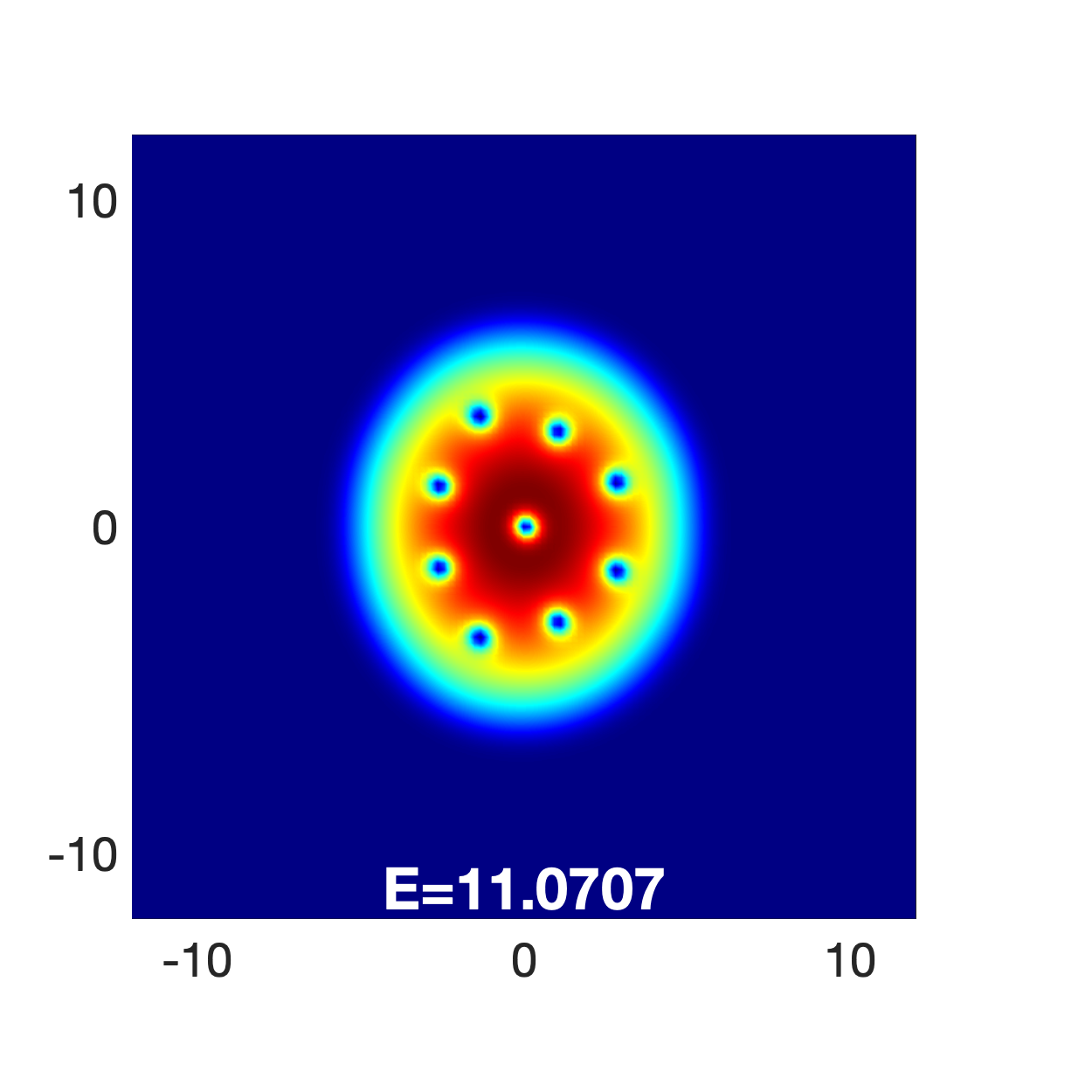} 
	\end{subfigure}
	\hspace{-0.5cm} 
	\begin{subfigure}{0.24\textwidth}
		\includegraphics[height=\textwidth,width=\textwidth, clip, trim = 2.5cm 3cm 1cm 1cm]{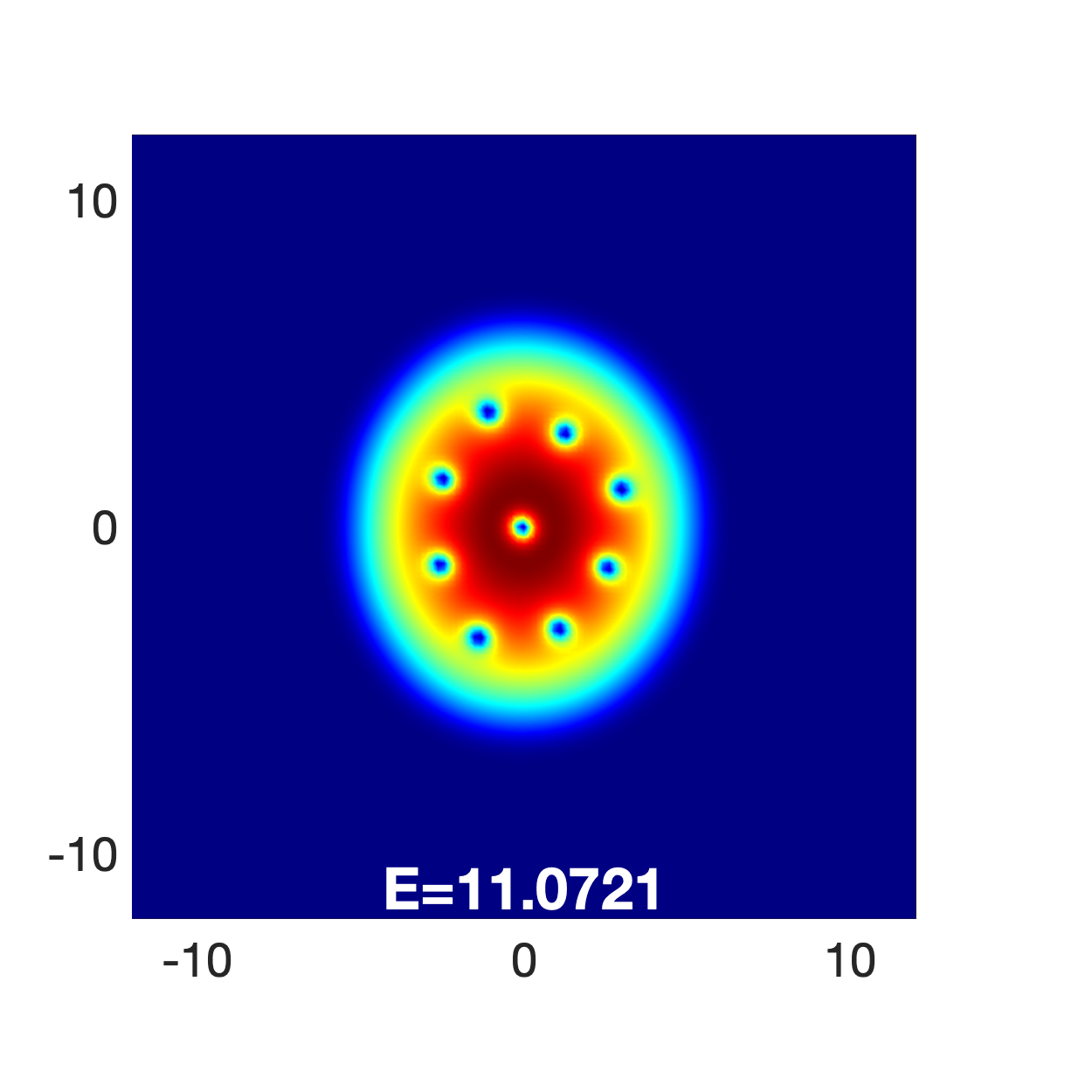} 
	\end{subfigure}
	\hspace{-0.5cm} 
	\begin{subfigure}{0.24\textwidth}
		\includegraphics[height=\textwidth,width=\textwidth, clip, trim = 2.5cm 3cm 0cm  1cm]{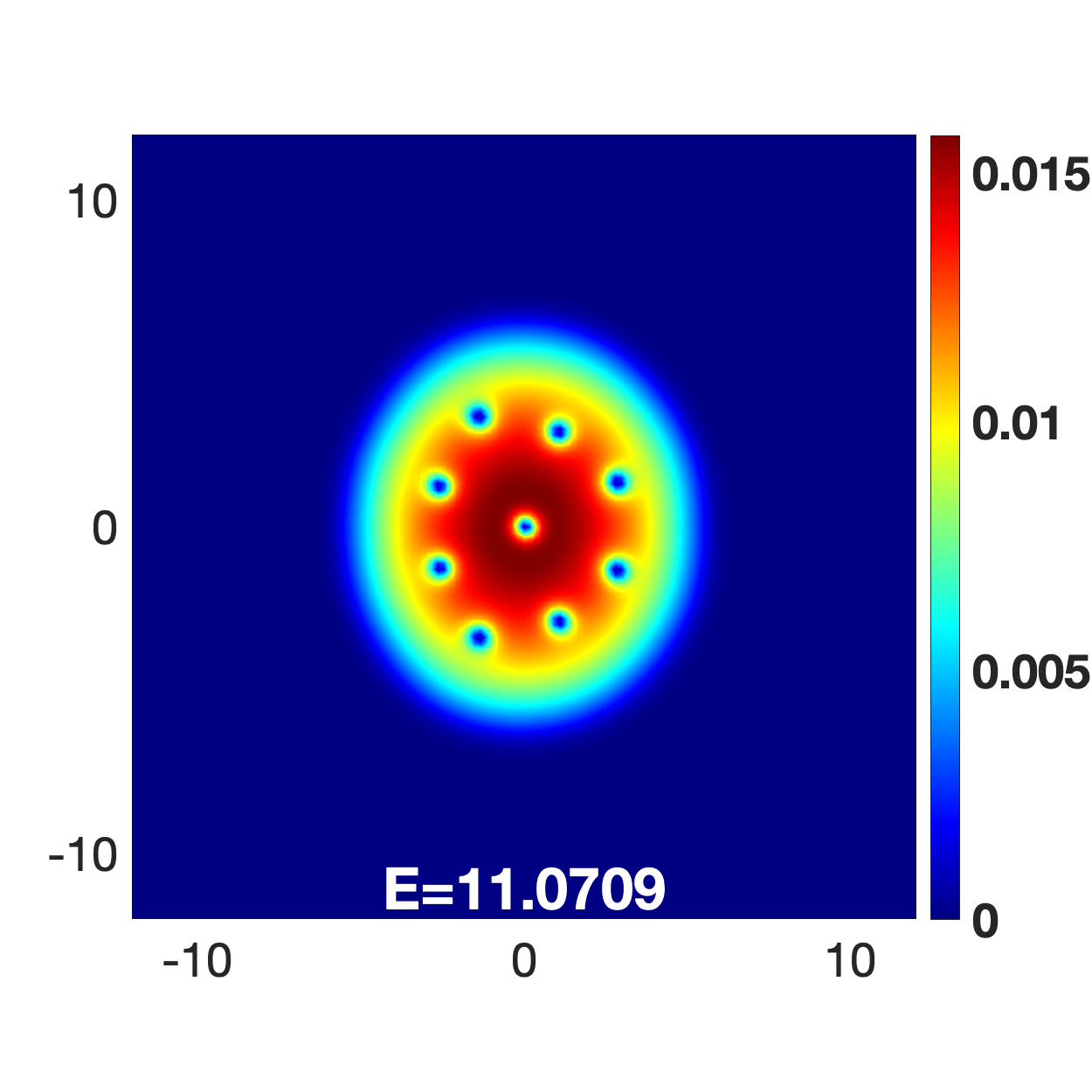}
	\end{subfigure}
    \captionsetup{skip=2pt} 
	\caption{$|\phi^{n+1}|^2$ for Example~\ref{Ex:rotAnis} with $\Omega=0.5$, $\tau_0=1/64$, $\tau_f=1/128$, $r=2$, $\tol= 2\times10^{-3}$. From left to right: Algorithm~\ref{Alg:EconcOrd1}, Algorithm~\ref{Alg:EconcOrd2}, Algorithm~\ref{Alg:EconcOrd1_Adap}, and Algorithm~\ref{Alg:EconcOrd2_Adap}.}\label{fig:2DOmg05}
\end{figure}
\begin{figure}[htbp]
	\centering
	\begin{subfigure}{0.24\textwidth}
		\includegraphics[height=\textwidth,width=\textwidth, clip, trim = 2.5cm 3cm 1cm 1cm]{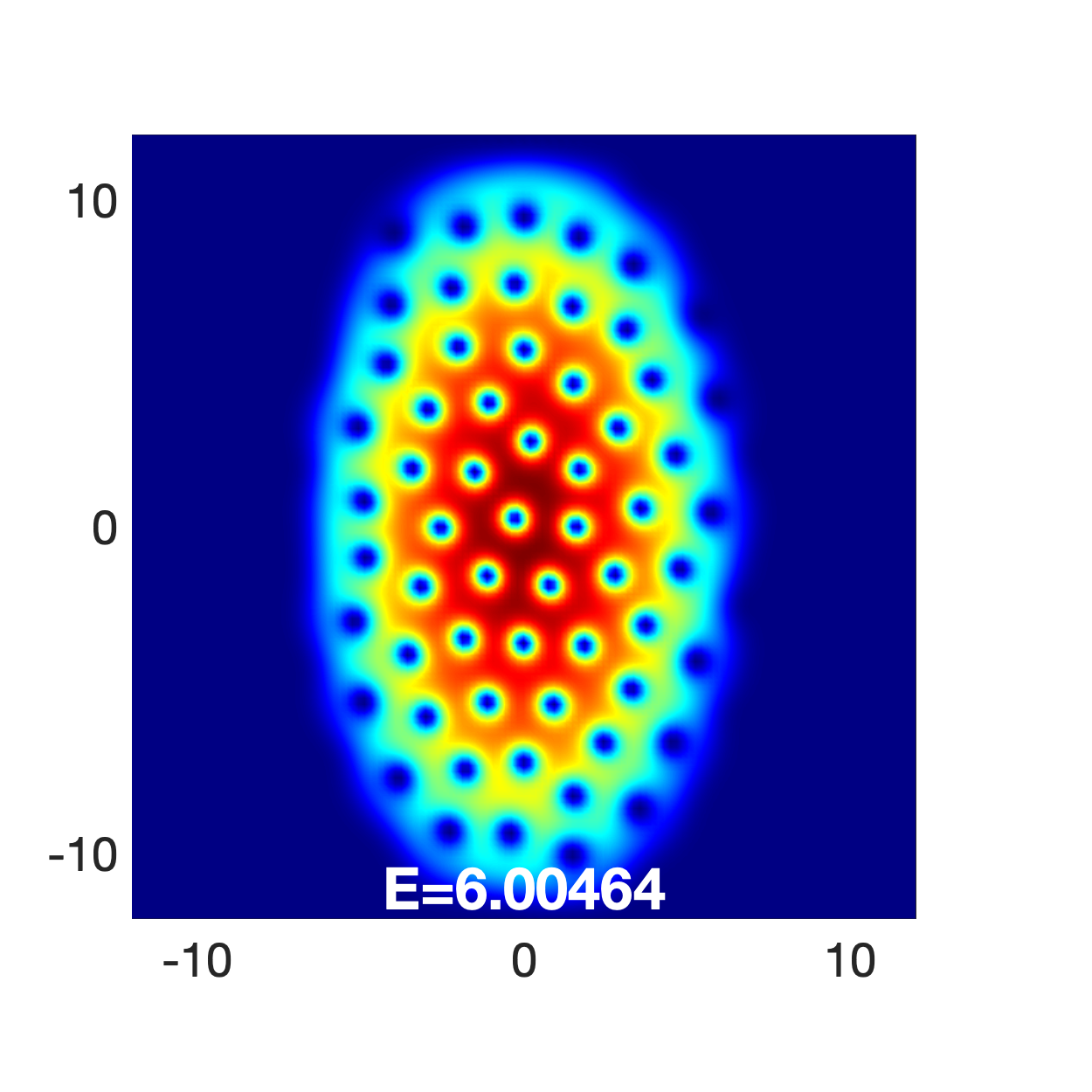}
	\end{subfigure}
	\hspace{-0.5cm} 
	\begin{subfigure}{0.24\textwidth}
		\includegraphics[height=\textwidth,width=\textwidth, clip, trim = 2.5cm 3cm 1cm 1cm]{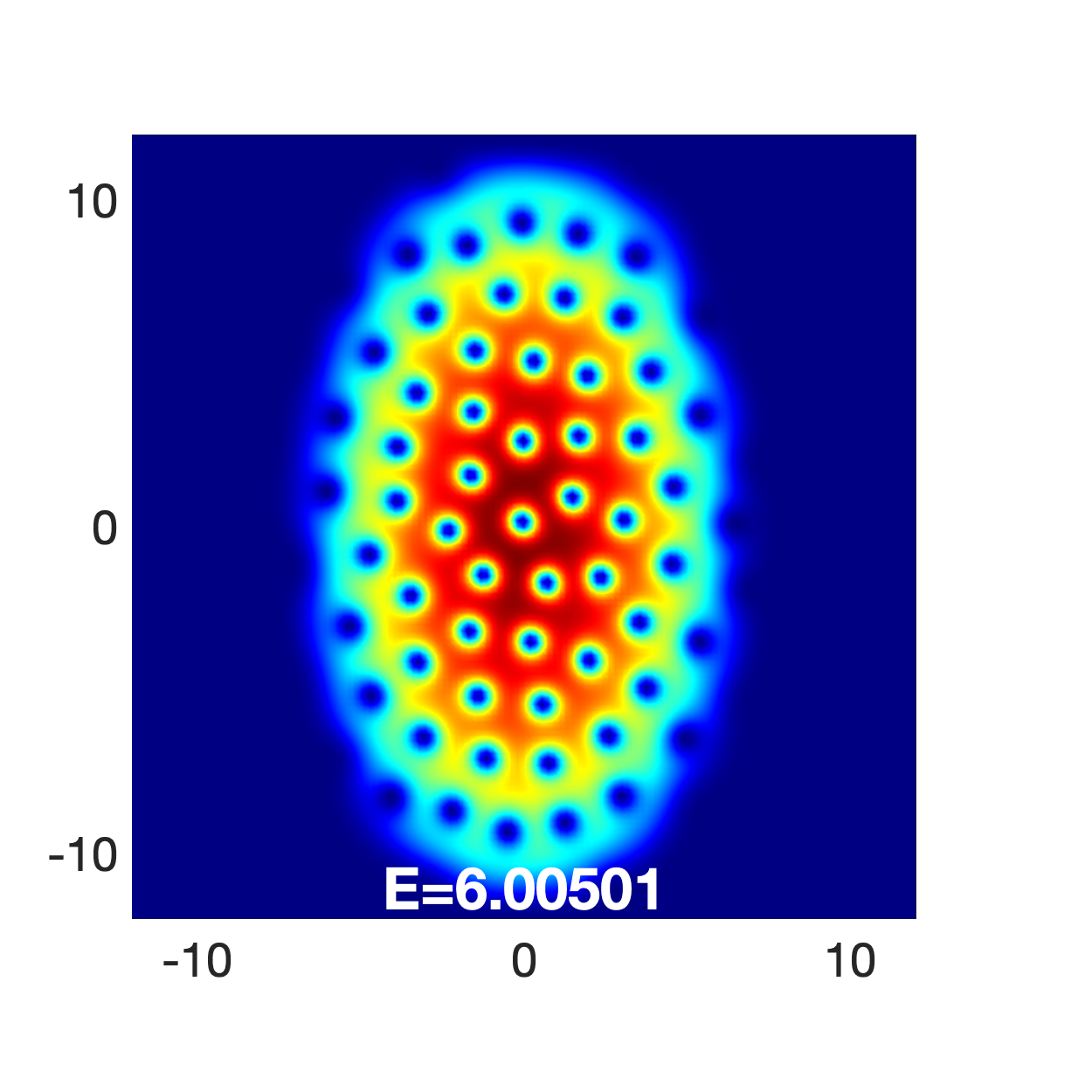} 
	\end{subfigure}
	\hspace{-0.5cm} 
	\begin{subfigure}{0.24\textwidth}
		\includegraphics[height=\textwidth,width=\textwidth, clip, trim = 2.5cm 3cm 1cm 1cm]{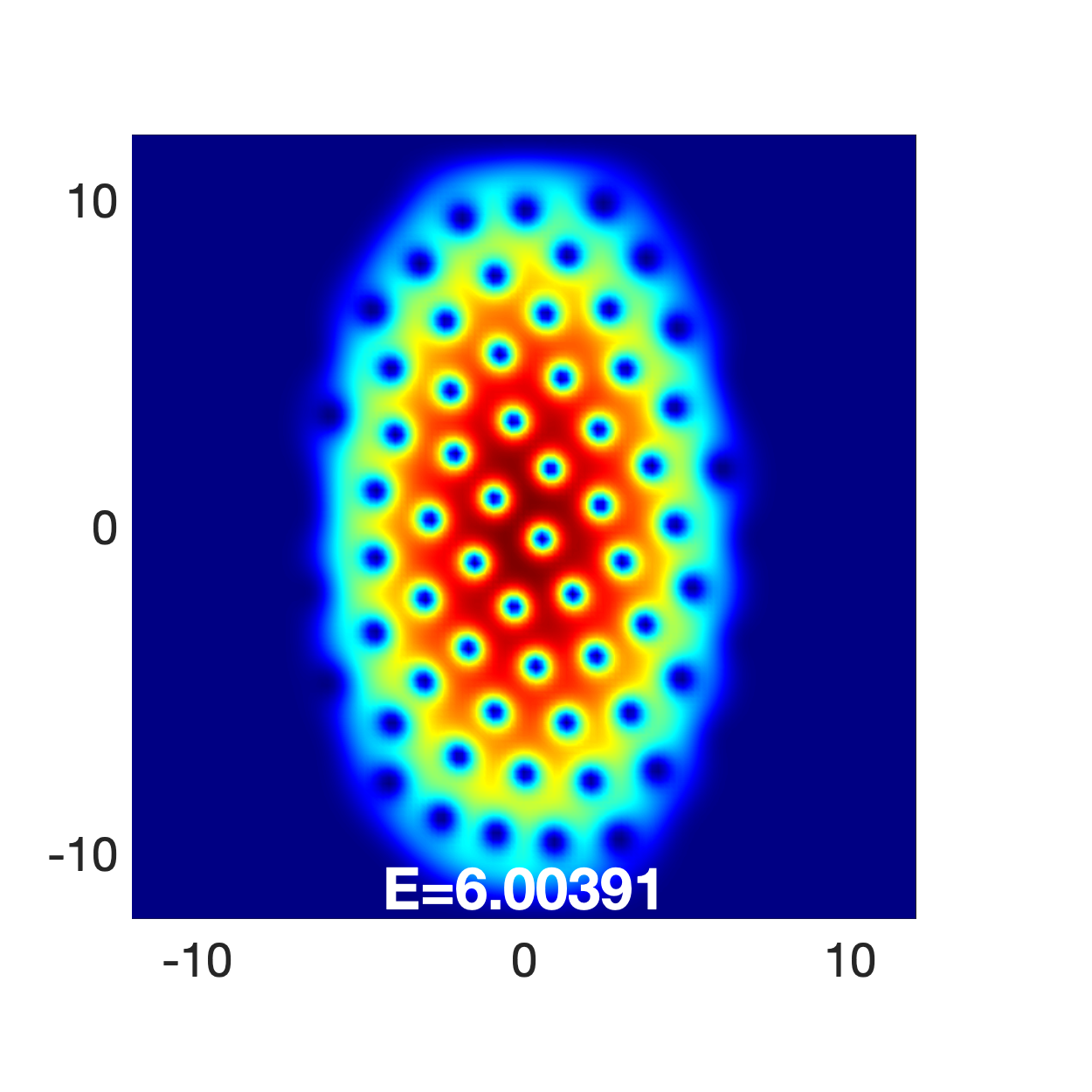} 
	\end{subfigure}
	\hspace{-0.5cm} 
	\begin{subfigure}{0.24\textwidth}
		\includegraphics[height=\textwidth,width=\textwidth, clip, trim = 2.5cm 3cm 0cm  1cm]{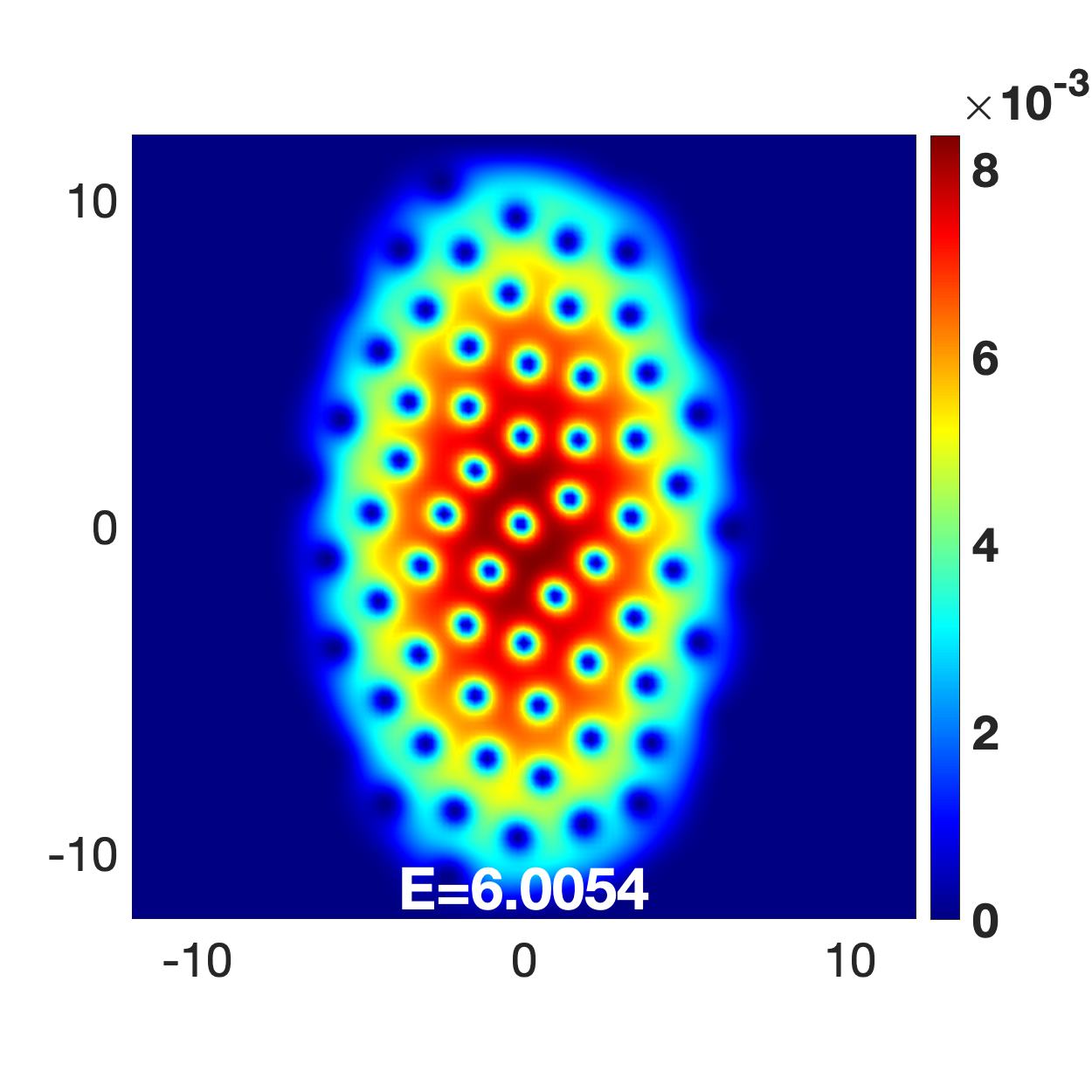}
	\end{subfigure}
    \captionsetup{skip=2pt} 
	\caption{$|\phi^{n+1}|^2$ for Example~\ref{Ex:rotAnis}  with $\Omega=0.9$, $\tau_0=1/64$, $\tau_f=1/128$, $r=2$, $\tol= 2\times10^{-3}$. From left to right: Algorithm~\ref{Alg:EconcOrd1}, Algorithm~\ref{Alg:EconcOrd2}, Algorithm~\ref{Alg:EconcOrd1_Adap}, and Algorithm~\ref{Alg:EconcOrd2_Adap}}\label{fig:2DOmg09}
\end{figure}

\begin{figure}[htbp]
	\centering
	\includegraphics[width=0.4\textwidth]{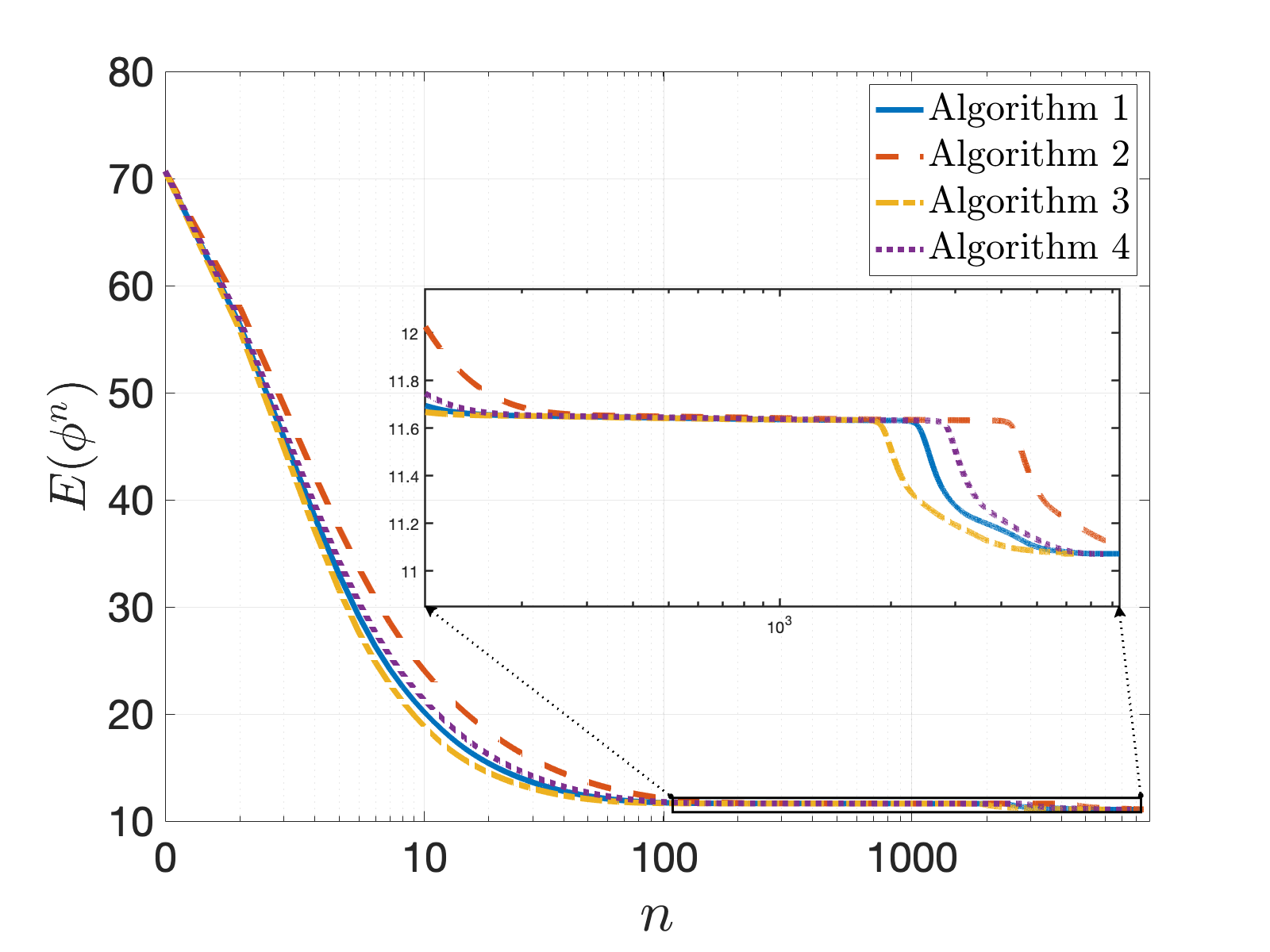}
	\includegraphics[width=0.4\textwidth]{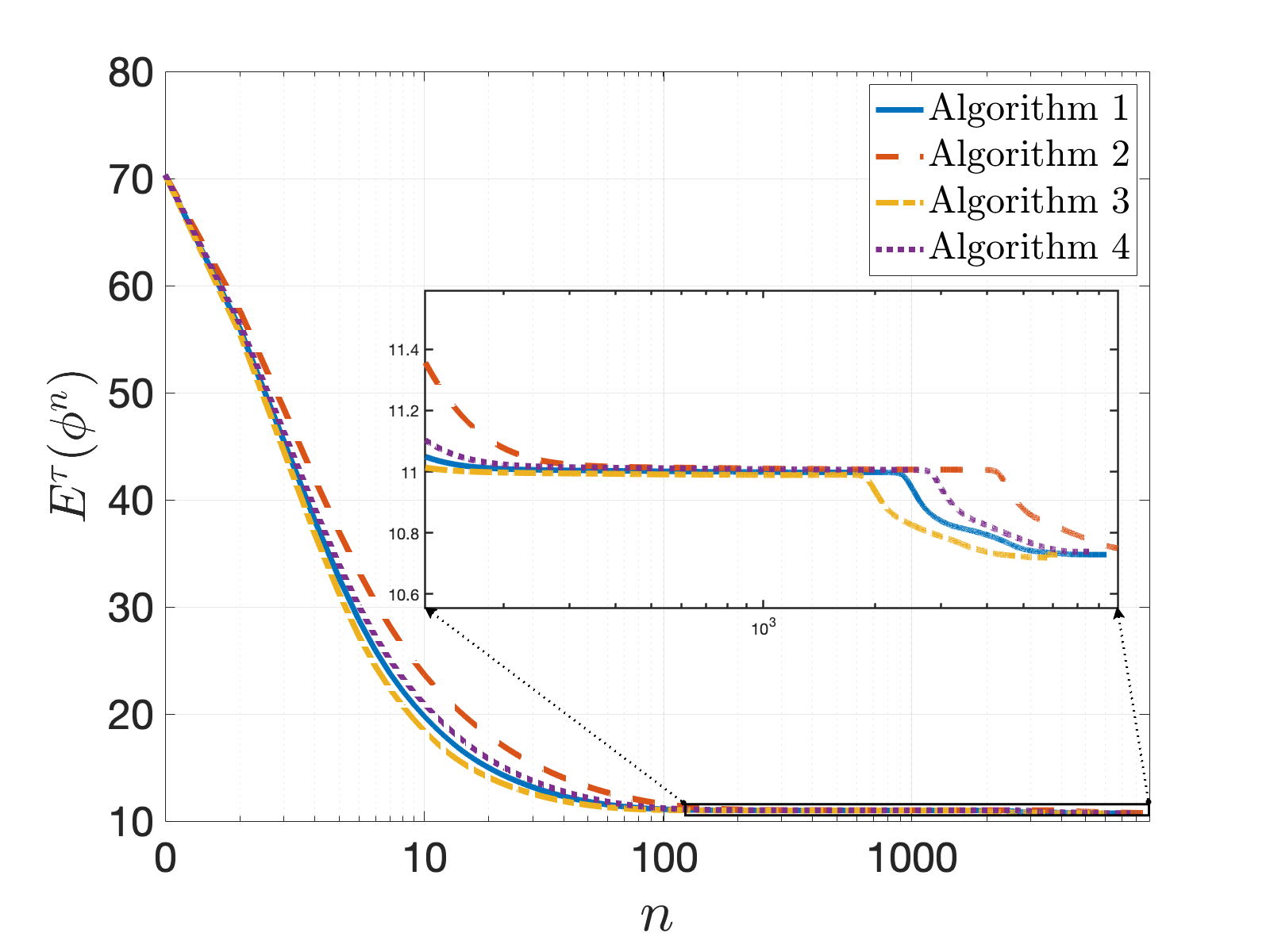}
      \captionsetup{skip=2pt} 
	\caption{Approximated energy for Example~\ref{Ex:rotAnis} using different algorithms with $\Omega=0.5$, $\tau_0=1/64$, $\tau_f=1/128$, $r=2$, $\tol=2\times 10^{-3}$.}
	\label{fig:2drotEomg05}
\end{figure}
\begin{figure}[htbp]
	\centering
	\includegraphics[width=0.4\textwidth]{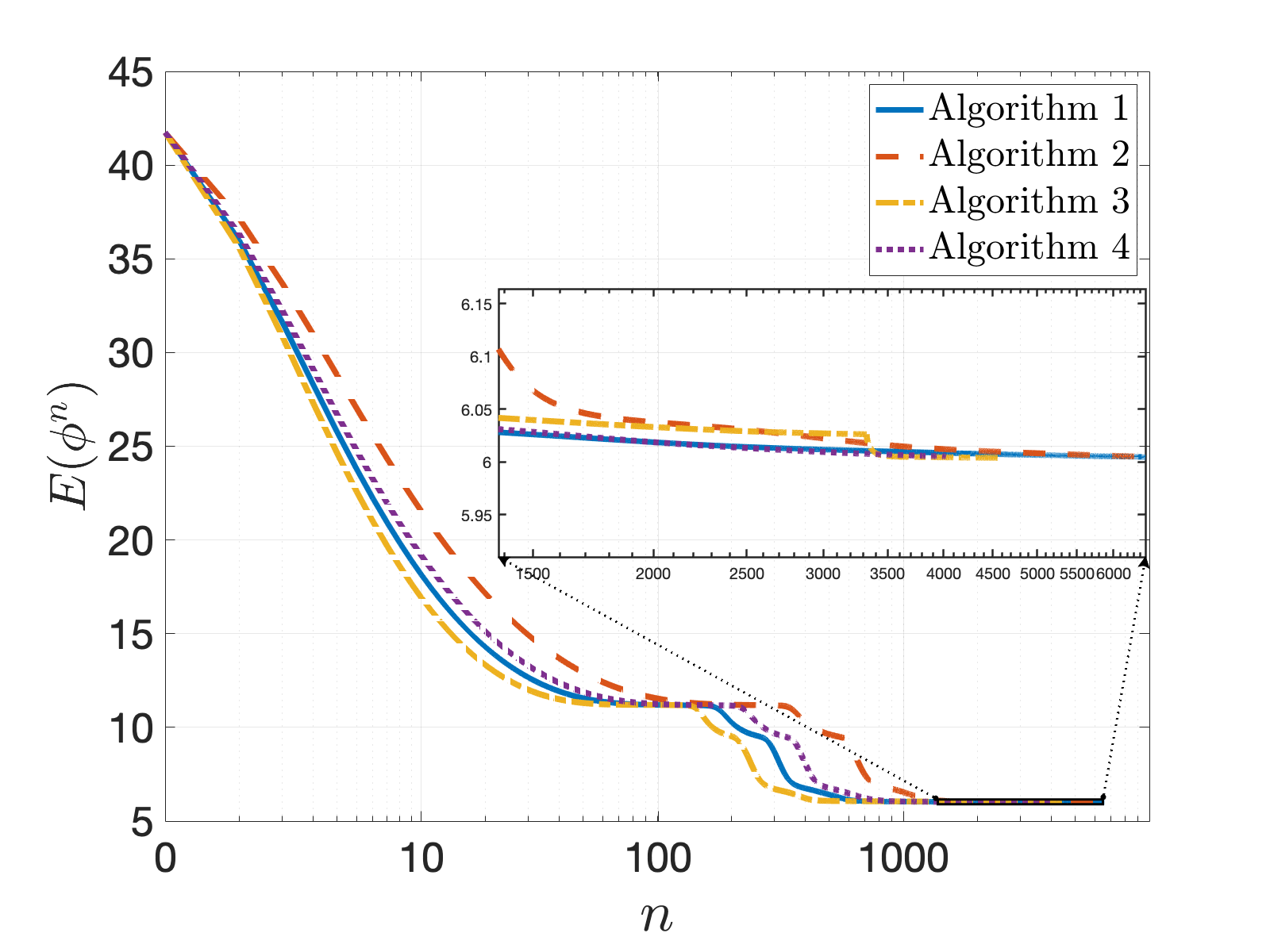}
	\includegraphics[width=0.4\textwidth]{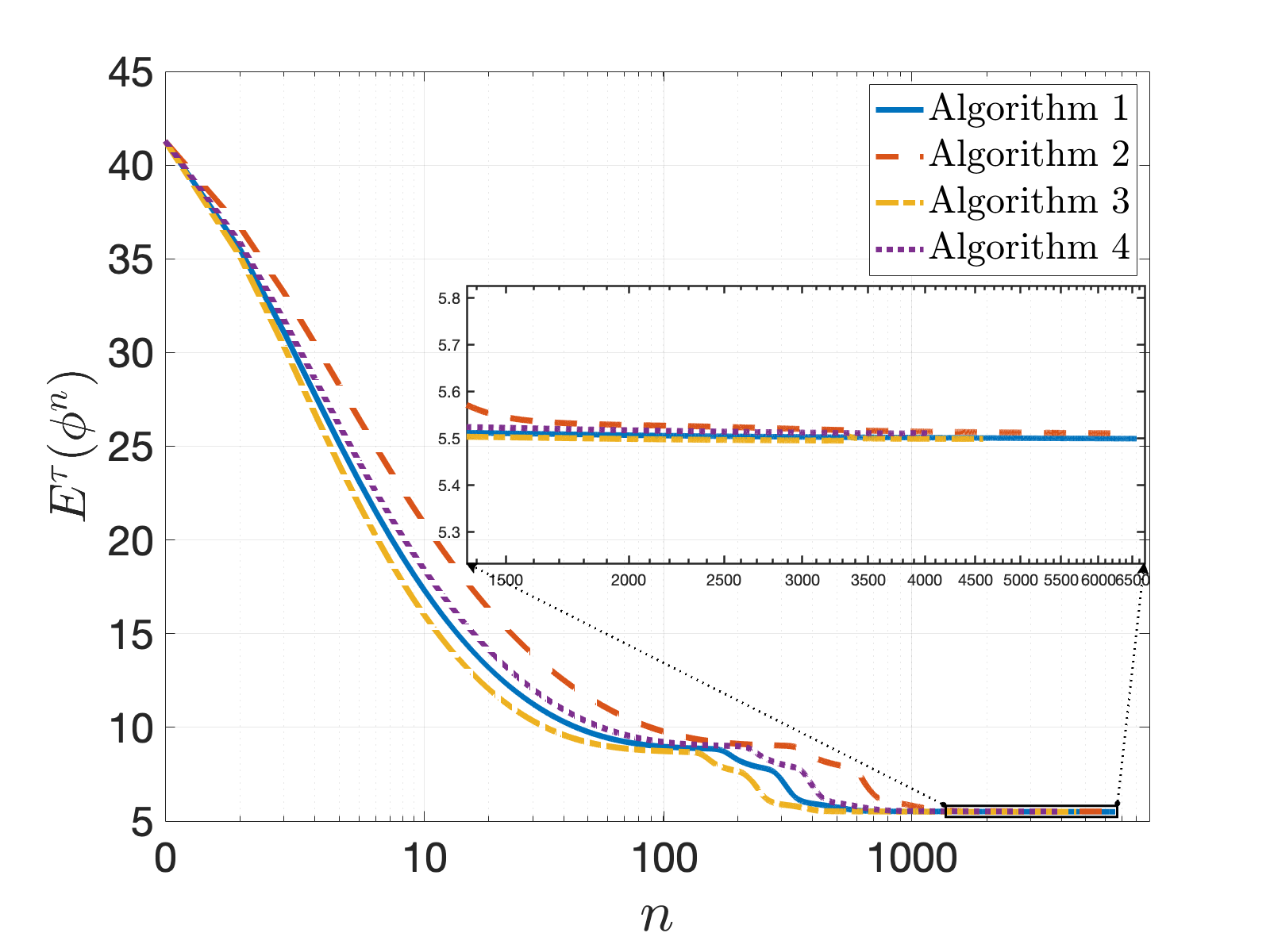}
      \captionsetup{skip=2pt} 
	\caption{Approximated energy for Example~\ref{Ex:rotAnis} using different algorithms with $\Omega=0.9$, $\tau_0=1/64$, $\tau_f=1/128$, $r=2$, $\tol=2\times 10^{-3}$.}
	\label{fig:2drotEomg09}
\end{figure}

\begin{figure}[h!]
	\centering
	\begin{subfigure}{0.4\textwidth}
		\includegraphics[width=\textwidth,clip, trim = 0cm 0cm 0cm 1cm]{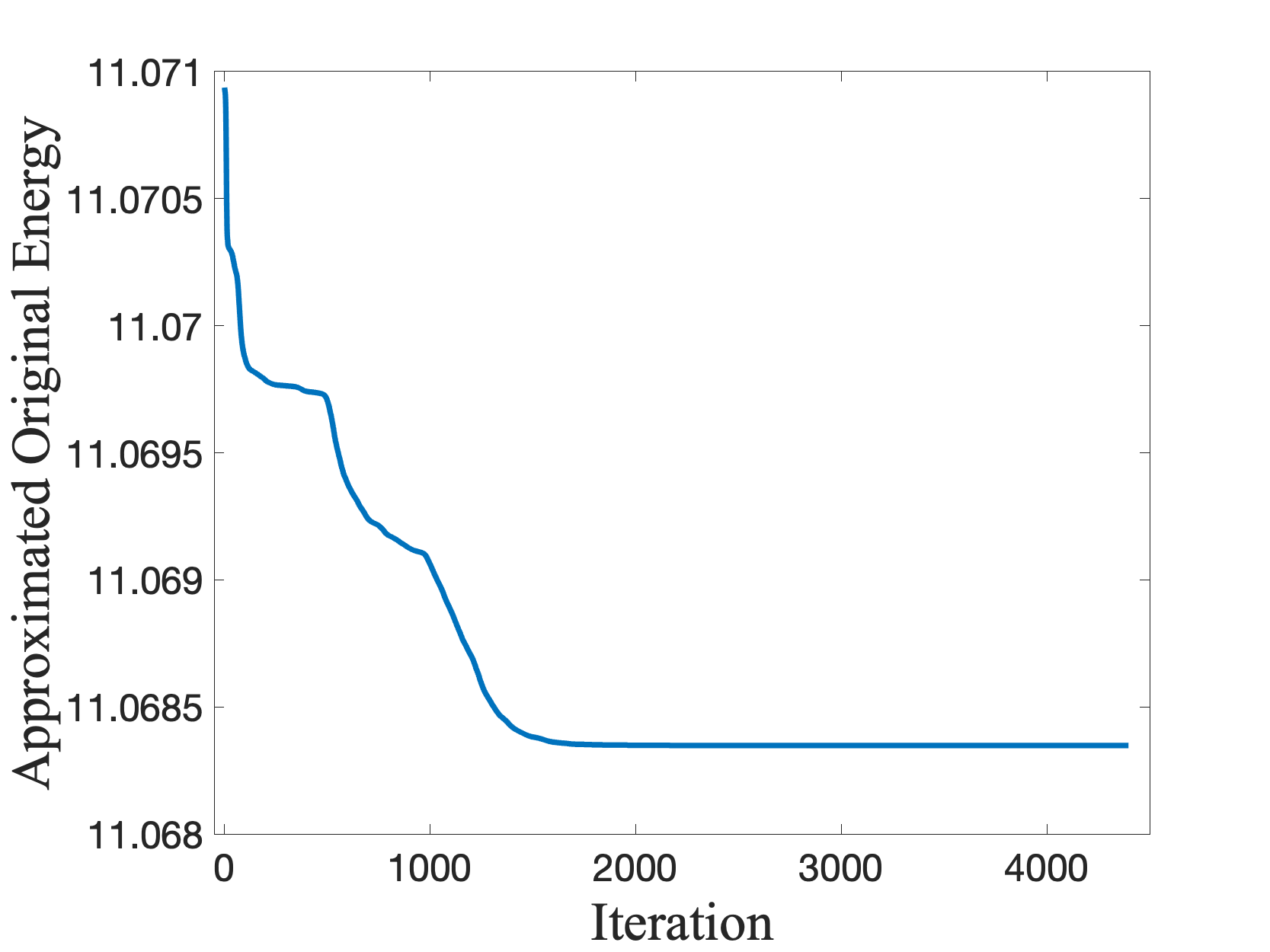}
	\end{subfigure}
	\begin{subfigure}{0.4\textwidth}
		\includegraphics[height=0.75\textwidth,width=0.9\textwidth, clip, trim = 2.5cm 3cm 0cm 2cm]{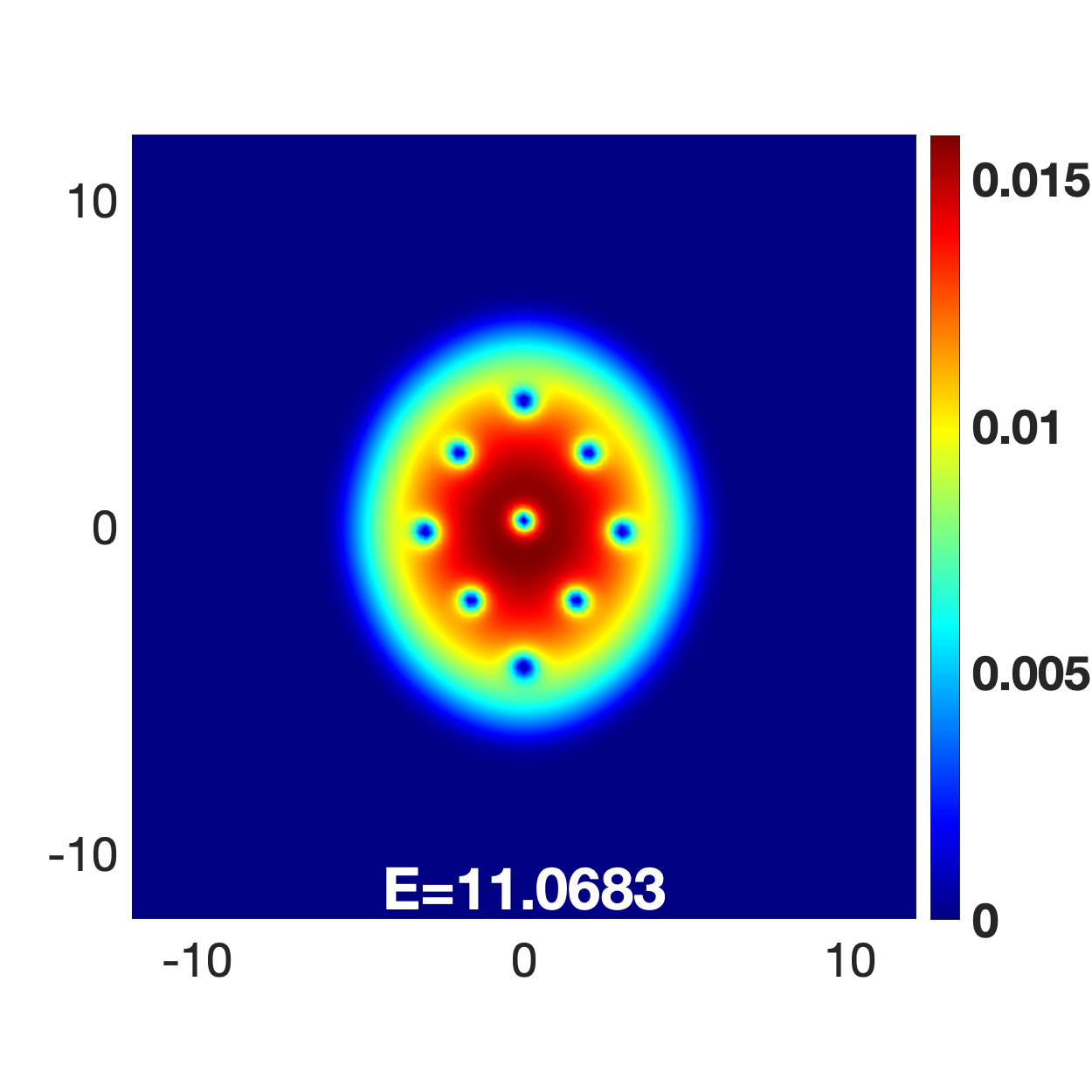}
	\end{subfigure}
  \captionsetup{skip=2pt} 
	\caption{Energy evolution (LEFT) and the final state of \( |\phi|^2 \) (RIGHT) for Example~\ref{Ex:rotAnis} with \( \Omega = 0.5 \) obtained using the RLBFGS algorithm.}
	\label{fig:Omg05rlbfgs}
\end{figure}

\section{Conclusion}\label{sec:conclusion}
We have introduced two relaxed formulations of the Gross--Pitaevskii energy functional, achieving first- and second-order accuracy in the relaxation parameter \( \tau \). Through rigorous theoretical analysis, it was shown that the relaxed functionals converge to the original functional as \( \tau \to 0 \), while their concavity facilitates optimization and guarantees energy dissipation during normalization. To solve the resulting optimization problems, energy-dissipative algorithms were developed using sequential linear programming, with their stability rigorously established. Furthermore, an adaptive \( \tau \) strategy was proposed, enabling a dynamic balance between accuracy and computational efficiency, which significantly enhances performance. Numerical experiments demonstrated the stability, convergence, and energy dissipation of the proposed methods, confirming their reliability and effectiveness.

The framework presented in this work offers a robust and efficient approach for computing the ground state of Bose--Einstein condensates (BECs). Its adaptability and theoretical rigor suggest strong potential for extension to more complex physical systems and broader optimization problems, making it a valuable tool for future research.

\bibliographystyle{siamplain}
\bibliography{BEC_concave}

\appendix
\section{Lemmas on the existence of a local minimizer for the first-order relaxed energy functional~\eqref{eng_concOrd1}}\label{sec:append}

To establish the existence of a local minimum for the relaxed problem \eqref{prob_relxOrd1}, we first prove that the energy functional \( E^{1,\tau}(\phi) \), defined in \eqref{eng_concOrd1}, satisfies \( E^{1,\tau}(\phi) \geq E^{1,\tau}(|\phi|) \) for all \(\phi \in \mathbb{S}\). This result is stated in the following lemma:
\begin{lemma}\label{lem:EconvOrd1absphi}
	For \(\tau > 0\),  
	\(\phi \in \mathbb{S}\), the energy functional in \eqref{eng_concOrd1} satisfies
	\[E^{1,\tau}(\phi)\ge E^{1,\tau}(|\phi|).\]
\end{lemma}
\begin{proof}
	The operator \( e^{\frac{\tau}{2} \Delta} \phi(\x) \) can be expressed  as
	\[
	e^{\frac{\tau}{2} \Delta} \phi(\x) = \int_{\mathcal{D}} G(\x - \y, \tau/2) \phi(\y) \, d\y,
	\]
	where \( G(\x - \y, t) \) is the heat kernel. In \( d \)-dimensional Euclidean space \( \mathbb{R}^d \), the heat kernel has the explicit form
	\[
	G(\x-\y, t) = \frac{1}{(4 \pi t)^{d/2}}  \exp\left( -\frac{|\x - \y|^2}{4t} \right).
	\]
	For \( \phi \) satisfying periodic boundary conditions on \( \partial \mathcal{D} \), the heat kernel still satisfies the non-negativity property
	\[
	G(\x - \y, t) \ge 0, \quad \forall \x, \y \in \mathcal{D}, \, t > 0.
	\]
	Using the positivity of the heat kernel, we derive
	\[
	\begin{aligned}
		\left| e^{\frac{\tau}{2} \Delta} |\phi(\x)| \right| &= \left| \int_{\Dc} G(\x - \y, \tau/2) |\phi(\y)| \, d\y \right| = \int_{\Dc} \left| G(\x - \y, \tau/2) \phi(\y) \right| \, d\y \\
		&\geq \left| \int_{\Dc} G(\x - \y, \tau/2) \phi(\y) \, d\y \right| = \left| e^{\frac{\tau}{2} \Delta} \phi(\x) \right|.
	\end{aligned}
	\]
	This implies
	\[
	-\frac{1}{2\tau} \int_{\Dc} \left| e^{\frac{\tau}{2} \Delta} |\phi| \right|^2 \, d\x \leq -\frac{1}{2\tau} \int_{\Dc} \left| e^{\frac{\tau}{2} \Delta} \phi \right|^2 \, d\x.
	\]
	Combining this with the definition of \( E^{1,\tau}(\phi) \), we conclude that
	\[
	E^{1,\tau}(\phi) \ge E^{1,\tau}(|\phi|),
	\]
	as required.
\end{proof}
The existence of a local minimum for \( E^{1,\tau}(\phi) \) on \( \mathbb{S} \) is therefore equivalent to the existence of a local minimum for \( E^{1,\tau}(|\phi|) \). Let \(\rho := |\phi|^2\). The functional \( E^{1,\tau}(|\phi|) \) can be expressed in terms of \(\rho\) as
\begin{equation} \label{eng_concOrd1convx}  
	E^{1,\tau}(\sqrt{\rho}) := \frac{1}{2\tau} + \int_{\mathcal{D}} \left( -\frac{1}{2\tau} \left| e^{\frac{\tau}{2} \Delta} \sqrt{\rho} \right|^2 + V(\x) \rho + \frac{\beta}{2} |\rho|^2  \right) \, d\x.
\end{equation}
From the lemma above, we obtain the equivalence
\[
\min_{\phi \in \mathbb{S}} E^{1,\tau}(\phi) = \min_{\rho \in \mathcal{M}} E^{1,\tau}(\sqrt{\rho}),
\]
where the feasible set \(\mathcal{M}\) is defined as
\begin{equation}\label{fesset_Ord1}
	\mathcal{M} := \left\{ \rho \in L^2(\mathcal{D}) \,\middle|\, \rho \ge 0, \,\, \int_{\mathcal{D}} \rho \, d\x = 1, \,\, E^{1,\tau}(\sqrt{\rho}) < \infty \right\}.
\end{equation}
Next, we present some key properties of \( E^{1,\tau}(\sqrt{\rho}) \).

\begin{lemma}\label{lem:convex}
For \( \tau > 0 \), under Assumption~\ref{assum:BEC}, the energy functional \( E^{1,\tau}(\sqrt{\rho}) \), defined in \eqref{eng_concOrd1convx}, is positive, coercive, and convex on the feasible set \( \mathcal{M} \) given in \eqref{fesset_Ord1}.
\end{lemma}
\begin{proof}
From Assumption~\ref{assum:BEC}, we have \( V(\x) \geq 0 \). Hence,
	\[
	E^{1,\tau}(\sqrt{\rho}) \geq \frac{1}{2\tau} - \int_{\mathcal{D}} \frac{1}{2\tau} \left| e^{\frac{\tau}{2} \Delta} \sqrt{\rho} \right|^2 \, d\x + \frac{\beta}{2} \int_{\mathcal{D}} \rho^2 \, d\x.
	\]
 Since \( e^{\frac{\tau}{2}\Delta} \) is a contraction semigroup on \( L^2(\mathcal{D}) \), it follows that
\[
	\frac{1}{2\tau} - \int_{\mathcal{D}} \frac{1}{2\tau} \left| e^{\frac{\tau}{2} \Delta} \sqrt{\rho} \right|^2 \, d\x \geq \frac{1}{2\tau} - \frac{1}{2\tau} \|\sqrt{\rho}\|_2^2 = 0.
	\]
	Combining this with \(\beta > 0\), we conclude that
	\[
	E^{1,\tau}(\sqrt{\rho}) \ge \frac{\beta}{2} \int_{\mathcal{D}} \rho^2 \, d\x > 0,
	\]
which establishes the positivity and coercivity on \( \mathcal{M} \).

To prove convexity, we note that all terms in \( E^{1,\tau}(\sqrt{\rho}) \), except for the first integral, are either constant or linear/quadratic in \( \rho \) and hence convex. For the integral term in \eqref{eng_concOrd1convx}, consider \(\rho_1, \rho_2 \in L^2(\Dc)\) with \(\rho_1, \rho_2 \geq 0\) and define \(\rho := \theta \rho_1 + (1-\theta)\rho_2\) for \(0 < \theta < 1\). Using the convolution representation of \(e^{\frac{\tau}{2}\Delta}\sqrt{\rho}\), we have
	\begin{align*}
		&\int_{\Dc} \left| e^{\frac{\tau}{2} \Delta} \sqrt{\rho} \right|^2\, d\x
		\\ &= \int_{\Dc}\int_{\Dc} G(\x-\y,\tau) \sqrt{\theta\rho_1(\y) + (1-\theta)\rho_2(\y)} \sqrt{\theta\rho_1(\x) + (1-\theta)\rho_2(\x)} \, d\y d\x.
	\end{align*}  
	On the other hand,  
	\begin{align*}
		&\theta \int_{\Dc} \left| e^{\frac{\tau}{2} \Delta} \sqrt{\rho_1} \right|^2\, d\x + (1-\theta) \int_{\Dc} \left| e^{\frac{\tau}{2} \Delta} \sqrt{\rho_2} \right|^2\, d\x \\
		&= \int_{\Dc} \int_{\Dc} G(\x-\y,\tau) \left( \theta \sqrt{\rho_1(\y)} \sqrt{\rho_1(\x)} + (1-\theta) \sqrt{\rho_2(\y)} \sqrt{\rho_2(\x)} \right) \, d\y d\x.
	\end{align*}  
	Using the nonnegativity of \( \rho \), \( \rho_1 \), and \( \rho_2 \), along with \( 0 < \theta < 1 \), we obtain
	\begin{align*}
		&\left(\theta \rho_1(\y) + (1-\theta) \rho_2(\y)\right)\left(\theta \rho_1(\x) + (1-\theta) \rho_2(\x)\right)\\
		&- \left( \theta \sqrt{\rho_1(\y)} \sqrt{\rho_1(\x)} + (1-\theta) \sqrt{\rho_2(\y)} \sqrt{\rho_2(\x)} \right)^2  \\
		&= \theta(1-\theta) \left( \rho_1(\y) \rho_2(\x) + \rho_1(\x) \rho_2(\y) \right)-2\theta(1-\theta) \sqrt{\rho_1(\y)\rho_1(\x)\rho_2(\y)\rho_2(\x)} \geq 0.
	\end{align*}
	Combining with \( G(\x-\y, \tau) \geq 0 \) on \( \Dc \), it follows that
	\[
	\int_{\Dc} \left| e^{\frac{\tau}{2} \Delta} \sqrt{\rho} \right|^2\, d\x \geq \theta \int_{\Dc} \left| e^{\frac{\tau}{2} \Delta} \sqrt{\rho_1} \right|^2\, d\x + (1-\theta) \int_{\Dc} \left| e^{\frac{\tau}{2} \Delta} \sqrt{\rho_2} \right|^2\, d\x.
	\]
	This implies  
	\[
	-\frac{1}{2\tau} \int_{\Dc} \left| e^{\frac{\tau}{2} \Delta} \sqrt{\rho} \right|^2\, d\x \leq -\frac{\theta}{2\tau} \int_{\Dc} \left| e^{\frac{\tau}{2} \Delta} \sqrt{\rho_1} \right|^2\, d\x - \frac{1-\theta}{2\tau} \int_{\Dc} \left| e^{\frac{\tau}{2} \Delta} \sqrt{\rho_2} \right|^2\, d\x.
	\]  
	Thus, the  convexity of the energy \( E^{1,\tau}(\sqrt{\rho}) \) is established.  
	The proof is complete.
\end{proof}
We now demonstrate the existence and uniqueness of the local minimum solution for \( E^{1,\tau}(\sqrt{\rho}) \) on the feasible set \(\mathcal{M}\).
\begin{lemma}\label{lem:exisAbsPhiOrd1}
	For \(\tau > 0\) and under Assumption~\ref{assum:BEC}, there exists a unique solution \(\rho \in \mathcal{M}\) to the minimization problem
	\[
	\min_{\rho \in \mathcal{M}} E^{1,\tau}(\sqrt{\rho}).
	\]
	where \(\mathcal{M}\) is defined in \eqref{fesset_Ord1} and \( E^{1,\tau}(\sqrt{\rho}) \) is given in \eqref{eng_concOrd1convx}.
\end{lemma}

\begin{proof}
	Let \( \{\rho_n\} \subset \mathcal{M} \) be a minimizing sequence for the functional \( E^{1,\tau}(\sqrt{\rho}) \). Since \( L^2(\mathcal{D}) \)  is reflexive, the Eberlein--\v{S}mulian theorem \cite[p.~430]{DunfSchw} guarantees the existence of a subsequence, still denoted by \( \{\rho_n\} \), such that \( \{\rho_n\} \) converges weakly in \( L^2(\mathcal{D})\) to  \(\rho^* \), i.e.,
	\[
	\rho_n \rightharpoonup \rho^* \quad \text{weakly in } L^2(\mathcal{D}).
	\]
	From the weak convergence of \( (\rho_n) \), it follows that
	\[
	\int_{\mathcal{D}} \rho_n \, d\x \to \int_{\mathcal{D}} \rho^* \, d\x \quad \text{as } n \to \infty.
	\]
	On the other hand, \(\rho_n\ge0\) implies \(\rho_*\ge0\) and hence \(\rho_*\in\mathcal{M}\).
	Since \( V(\x) \in L^\infty(\mathcal{D}) \) and \( V(\x) \geq 0 \), we deduce
	\[
	\int_{\mathcal{D}} V(\x) \rho_n \, d\x \to \int_{\mathcal{D}} V(\x) \rho^* \, d\x \quad \text{as } n \to \infty.
	\]
	The weak lower semicontinuity of the \( L^2 \)-norm, combined with \(\beta > 0\), implies
	\[
	\frac{\beta}{2} \int_{\mathcal{D}} (\rho^*)^2 \, d\x \leq \liminf_{n \to \infty} \frac{\beta}{2} \int_{\mathcal{D}} \rho_n^2 \, d\x.
	\]
	For \( -\frac{1}{2\tau}\int_{\mathcal{D}} \left| e^{\frac{\tau}{2} \Delta} \sqrt{\rho} \right|^2d\x \), Lemma \ref{lem:convex} shows it is convex in \(\rho\in\mathcal{M}\) and \(\mathcal{M}\) is a closed convex set. Then \( -\frac{1}{2\tau}\int_{\mathcal{D}} \left| e^{\frac{\tau}{2} \Delta} \sqrt{\rho} \right|^2d\x \) is weakly lower semi-continuous if and only if it is lower semi-continuous. Noticing
	if \(\rho_n\to\rho \in\mathcal{M}\) in \(L^2(\Dc)\), \(\sqrt{\rho_n}\to\sqrt{\rho}\) in \(L^2(\Dc)\), and the lower semi-continuity as well as the weakly lower semi-continuity of \( -\frac{1}{2\tau}\int_{\mathcal{D}} \left| e^{\frac{\tau}{2} \Delta} \sqrt{\rho} \right|^2d\x \) follows.
	
	Combining these results, we obtain
	\[
	E^{1,\tau}(\sqrt{\rho^*}) \leq \liminf_{n \to \infty} E^{1,\tau}(\sqrt{\rho_n}).
	\]
	Thus,   \( \rho^* \in \mathcal{M} \) is a minimizer of \(E^{1,\tau}\) on \(\mathcal{M}\).  The convexity of \( E^{1,\tau}(\sqrt{\rho}) \) further guarantees the uniqueness of the minimizer.
\end{proof}

\end{document}